\documentclass{amsart}
\usepackage[T1]{fontenc}
\usepackage[sorting=nty,natbib=true,backend=biber,style=numeric]{biblatex}
\usepackage[margin=1in]{geometry}
\usepackage{bm}% bold math
\usepackage{xcolor}% Text highlighting
\usepackage{booktabs}
\usepackage{amssymb}
\usepackage{amsthm}
\usepackage{mathrsfs}
\usepackage{amsfonts}
\usepackage{amsmath}
\usepackage{xcolor}
\usepackage{tikz}
\usepackage{tkz-graph}
\usepackage{enumitem}
\usepackage{caption}
\usepackage{subcaption}

\usepackage{url, hyperref}
\hypersetup{citecolor=blue, linkcolor=blue, colorlinks=true}
\usepackage{setspace}
\usepackage{graphicx}
\usepackage{mathtools}
\usepackage{calc}
\usepackage{tikz}
\usetikzlibrary{quantikz}
\usepackage{array}
\newcolumntype{P}[1]{>{\centering\arraybackslash}p{#1}}

\usepackage{longtable}
\tikzstyle{v}=[circle, draw, fill=white,
inner sep=0pt, minimum width=4pt]
\newcommand{\dr}[1]{\rule[-3ex]{0pt}{0pt}\rule[4ex]{0pt}{0pt} \raisebox{\dimexpr-.5\height+.5\ht\strutbox\relax}{\tikz\draw#1;}}

\usepackage{color}
\usepackage{hyperref} % adds hyper links inside the generated pdf file
\hypersetup{
	colorlinks=true,       % false: boxed links; true: colored links
	linkcolor=blue,        % color of internal links
	citecolor=blue,        % color of links to bibliography
	filecolor=magenta,     % color of file links
	urlcolor=blue         
}

\newcommand{\Z}{\mathbb{Z}}

\newcommand{\id}{{\rm id}}

\newcommand{\rank}{\operatorname{rank}}

\newcommand{\allpp}{{\mathcal{A}}}

\newcommand{\graphf}{{\mathcal{G}}}

\newcommand{\type}{{\rm type}}
\newcommand{\numfact}[2]{\mathcal{N}_{#1}({#2})}
\newcommand{\numftree}[1]{N({#1})}

\newcommand{\partperm}[2]{({#1},{#2})}

\newcommand{\len}{\operatorname{len}}

\newcommand{\cycle}{c}
\newcommand{\parts}{\ell}

\newcommand{\lattice}{\mathscr{L}}

\newcommand{\ind}{\operatorname{in}}
\newcommand{\outd}{\operatorname{out}}

\DeclareMathOperator{\aut}{aut}

\usepackage{fancyhdr}
\usepackage{xcolor}
\usepackage{soul}

\theoremstyle{plain}
\newtheorem{theorem}{Theorem}[section]
\newtheorem{corollary}[theorem]{Corollary}
\newtheorem{proposition}[theorem]{Proposition}
\newtheorem{lemma}[theorem]{Lemma}
\newtheorem{conjecture}[theorem]{Conjecture}

\theoremstyle{definition}
\newtheorem{definition}[theorem]{Definition}
\newtheorem{example}[theorem]{Example}
\newtheorem{algorithm}[theorem]{Algorithm}

\theoremstyle{remark}

\newtheorem{remark}[theorem]{Remark}

\title{Minimal transitive factorizations \\ supported on quasi-threshold graphs}

\author{Cordelia Yuqiao Li}
\address{Department of Mathematics, University of Washington, Seattle, WA 98195}
\email{yuqiaoli@uw.edu}

\author{Ricky Ini Liu}
\address{Department of Mathematics, University of Washington, Seattle, WA 98195}
\email{riliu@uw.edu}
\thanks{The second author is supported by a Simons Foundation Travel Support for Mathematicians award.}

\begin{document}

\begin{abstract}
We study the number of minimal transitive factorizations of the identity permutation in $S_n$ into transpositions supported on a quasi-threshold graph. We show that this number is always divisible by $(2n-2)!/n!$, which is the factorization count for a star graph, as shown by Irving and Rattan. To prove this, we give a combinatorial formula for the number of such factorizations as a weighted sum over a subset of the $n^{n-2}$ \emph{factorization trees}, which are edge-weighted spanning trees satisfying certain flow constraints.
\end{abstract}
\maketitle
\section{Introduction}

A classical problem in combinatorics is to count the number of factorizations of a permutation in the symmetric group $S_n$ into transpositions (or other permutations) satisfying certain constraints. Hurwitz~\cite{Hurwitz_ori_paper} first studied this problem in the context of counting ramified coverings of Riemann surfaces, giving a formula for the number of minimal length transitive factorizations of a given permutation into transpositions, where \emph{transitive} means that the transpositions generate the entire group $S_n$. D\'enes \cite{Denes} later studied a special case of this when the given permutation is an $n$-cycle, showing that the number of minimal factorizations equals $n^{n-2}$, the number of labeled trees with $n$ vertices, using a combinatorial argument; Moszkowski \cite{moszkowski} later gave a direct bijection between these factorizations and labeled trees. Later, Goulden and Jackson~\cite{Goulden-Jackson} derived Hurwitz's formula using Lagrange inversion on generating functions and extended it to give a general expression for the number of transitive factorizations of any length in terms of the irreducible characters of $S_n$. On the bijective side, Bousquet-M\'elou and Schaeffer \cite{paper:planar_constellations} constructed a bijection between certain transitive factorizations and planar constellations, and Duchi, Poulalhon, and Schaeffer \cite{paper:Galaxy} later gave a fully bijective proof that recovers Hurwitz's formula. 

% This problem also arises naturally in theoretical nuclear physics \cite{paper:cresc_taylor} as a question related to topologically distinct holomorphic maps on the sphere. 

A related problem is to consider factorizations into transpositions when these transpositions are restricted to lie in the edge set of a given graph $G$, in which case we say the factorization is \emph{supported} on $G$. Let $\numfact{\sigma}{G}$ denote the number of minimal transitive factorizations of a permutation $\sigma \in S_n$ supported on $G$. Irving and Rattan \cite{min_star} give the following closed formula when $G$ is a star graph $K_{1,n-1}$, generalizing an earlier result of Pak \cite{paper:Pak}:
\begin{equation*}\label{Eq:star_formula}
    \numfact{\sigma}{K_{1,n-1}} = \frac{(n+m-2)!}{n!}l_1 \cdots l_m,
\end{equation*}
where $l_1, \dots, l_m$ are the cycle lengths of $\sigma$. When $\sigma$ is the identity permutation, this formula simplifies to \[\numfact{\id}{K_{1,n-1}} = \frac{(2n-2)!}{n!}.\]

Computer experimentation suggests that whenever the graph $G$ lies in the class of \emph{cographs} (graphs containing no induced path on four vertices), the count $\numfact{\id}{G}$ is also divisible by the quantity $(2n-2)!/n!$. In this paper, we consider the subclass of \emph{quasi-threshold graphs}, which are the comparability graphs of (naturally labeled) rooted forests. The main result of this paper is to prove this divisibility property for quasi-threshold graphs via the following theorem, which generalizes the result of Irving and Rattan for the identity permutation.
\begin{theorem}
    Let $G$ be a quasi-threshold graph on $n$ vertices. Then
    \[\numfact{\id}{G} = \frac{(2n-2)!}{n!}\sum_{T} w(T), \]
    where $T$ ranges over factorization trees that are weighted spanning trees of $G$, and $w(T)$ is the weight of $T$. 
\end{theorem}
Here, a factorization tree is an edge-weighted tree that satisfies certain flow inequalities; the definition of these trees and their weights can be found in Section~\ref{sec:factorization_tree}. (Factorization trees are also equivalent to \emph{genus 0 labeled floor diagrams} as described in work of Fomin and Mikhalkin \cite{floor_diagrams}.)

To prove this result, we first use Hurwitz moves to transform a factorization on a quasi-threshold graph into a concatenation of factorizations on several intermediate star graphs. We then use a variant of the cycle lemma by Dvoretzky and Motzkin \cite{paper:cycle_lemma_Dvoretzky_Motzkin}  to count the number of factorizations at each intermediate star. We then construct a factorization tree that captures the structure of how these star factorizations interact, and then we use the weight of the factorization tree to combine the results from the stars together and produce the final count. 

The outline of the paper is as follows. In Section~\ref{sec:preliminaries}, we discuss basic techniques for counting factorizations, including partitioned permutations and Hurwitz moves, as well as quasi-threshold graphs.
In Section~\ref{sec:lattice_walk}, we introduce a lattice walk method to visualize factorizations and prove a cycle lemma to count minimal length factorizations under a specific setup. In Section~\ref{sec:factorization_tree}, we define factorization trees and use them to prove our main result. Finally, in Section~\ref{sec:conclusion}, we conjecture the divisibility result for the larger class of cographs and provide some further discussion. The appendix contains computational data for $\numfact{\id}{G}$ for graphs on at most $5$ vertices.

\section{Preliminaries}\label{sec:preliminaries}
We begin with preliminary definitions related to factorizations in the symmetric group, introducing the notion of partitioned permutations and minimal transitive factorizations. We also introduce quasi-threshold graphs and discuss the Hurwitz action on factorizations supported on such graphs.

\subsection{Factorizations}\label{subsec:factorizations}
Let $S_n$ denote the symmetric group on $n$ letters, which we assume acts on $[n] = \{1, \dots, n\}$ unless otherwise specified. We denote the set of all transpositions in $S_n$ by $T_n$, writing $s_{ij}$ for the transposition $(i \; j)$ (usually with $i<j$) and $s_i$ for the simple transposition $(i \; i+1)$.

\begin{definition}
    A \emph{factorization} of $\sigma \in S_n$ (into transpositions) is an ordered tuple $f = (t_1, \dots, t_k)$ such that $\sigma = t_1 \cdots t_k$, and $t_i \in T_n$ for all $1 \leq i \leq k$.
    
    We say that $f$ is \emph{transitive} if the subgroup generated by $\{ t_1, \dots, t_k\}$ acts transitively on $[n]$. When $k$ is minimal among all transitive factorizations of $\sigma$, we say that $f$ is a \emph{minimal transitive factorization} of $\sigma$.
\end{definition}

Note that in $S_n$, a subgroup generated by transpositions acts transitively on $[n]$ if and only if it is the entire group. As we will see below in Corollary~\ref{cor:minlength}, the minimum length of a transitive factorization of $\sigma$ is $n+m-2$, where $m$ is the number of cycles of $\sigma$.

We will identify each transposition $s_{ij} \in T_n$ with the edge $\{i,j\}$ in the complete graph $K_n$ on vertex set $[n]$. Each factorization then yields a graph as follows.

\begin{definition}
    Let $f$ be a factorization of a permutation $\sigma \in S_n$. The graph $\graphf_f$ is the (simple) graph with vertex set $[n]$ and edges $\{i,j\}$ for each transposition $s_{ij}$ appearing in $f$.

    If $G$ is any graph with vertex set $[n]$, then we say that $f$ is \emph{supported} on $G$ if $\graphf_f \subseteq G$.
\end{definition}

Note that a factorization $f$ is transitive if and only if its corresponding graph $\graphf_f$ is connected.

\subsection{Prior enumerative results}

Let $\numfact{\sigma}{G}$ denote the number of minimal transitive factorizations $f$ of $\sigma$ supported on $G$. 
The following classical result due to Hurwitz \cite{Hurwitz_ori_paper} gives a formula for $\numfact{\sigma}{K_n}$.

\begin{theorem}[Hurwitz \cite{Hurwitz_ori_paper}] \label{thm:Hurwitz}
    Suppose $\sigma \in S_n$ has cycle lengths $l_1, \dots, l_m$. Then 
    \begin{equation*}\label{Eq:complete_graph_id}
    \numfact{\sigma}{K_{n}} = (n+m-2)! \cdot n^{m-3}\cdot \frac{l_1^{l_1+1}\cdots l_m^{l_m+1}}{l_1!\cdots l_m!}.
    \end{equation*}
    In particular, $\numfact{\id}{K_{n}} = n^{n-3}(2n-2)!$.
\end{theorem}

When $G$ is the star graph $K_{1,n-1}$, Irving and Rattan \cite{min_star} give the following closed formula for $\numfact{\sigma}{K_{1,n-1}}$, which surprisingly depends only on the cycle type of $\sigma$ and not the length of the cycle containing the center of the star.
\begin{theorem} [Irving--Rattan \cite{min_star}]
    Suppose $\sigma \in S_n$ has cycle lengths $l_1, \dots, l_m$. Then
    \[\numfact{\sigma}{K_{1,n-1}} = \frac{(n+m-2)!}{n!}l_1 \cdots l_m.\]
    In particular, $\numfact{\id}{K_{1,n-1}} = \frac{(2n-2)!}{n!}$.
\end{theorem}

Note that we have the following divisibility relation between these two quantities:
\[\frac{\numfact{\id}{K_{n}}}{\numfact{\id}{K_{1,n-1}}}= n^{n-2}(n-1)!.\]

\subsection{Partitioned permutations}\label{subsec:part_perm}

To study transitive factorizations $f$, it will be important for us to keep track of the connected components of $\mathcal{G}_{f}$.

Denote the lattice of all set partitions of $[n]$ (ordered by refinement) by $\Pi_n$, which has maximum element $\hat{1} = 1 2 \cdots n$, the partition with one block. Inspired by the construction of partial permutations in \cite{paper:partial_perm}, we define a partitioned permutation $\partperm{\sigma}{\pi}$ as follows. 

\begin{definition}\label{def:partperm}
    A \emph{partitioned permutation} $\partperm{\sigma}{\pi}$ is a pair consisting of a permutation $\sigma \in S_n$ and a set partition $\pi \in \Pi_n$ such that elements in the same cycle of $\sigma$ are in the same part of $\pi$.
    
    We denote the set of all partitioned permutations $\partperm{\sigma}{\pi}$ with $\sigma \in S_n$ and $\pi \in \Pi_n$ by $\allpp_n$.
\end{definition} 

For a fixed permutation $\sigma$, we can construct a partitioned permutation $(\sigma, \pi)$ by grouping together cycles of $\sigma$ to form the blocks of $\pi$. Alternatively, given a set partition $\pi$, we can construct a partitioned permutation $(\sigma, \pi)$ by defining permutations on each part of $\pi$, which combine to form $\sigma$. We will sometimes write a partitioned permutation by writing the permutation in cycle notation and adding bars separating the cycles into blocks.

\begin{example}
    The pair $(\sigma,\pi)$, where $\sigma = (1 \ 2 \ 3)(4 \ 5)(6)(7) \in S_7$ and $\pi = 1236\mid 45 \mid 7 \in \Pi_7$, is a partitioned permutation.  We may also write $(\sigma, \pi) = (1 \ 2 \ 3) (6)  \mid  (4 \ 5)  \mid  (7)$.
\end{example}

For two partitioned permutations $(\sigma, \pi), (\sigma', \pi') \in \allpp_n$, define the product 
\[(\sigma, \pi) \cdot (\sigma', \pi') = (\sigma \cdot \sigma', \pi \lor \pi') \in \allpp_n,\]
where $\sigma \cdot \sigma'$ is the usual product of two permutations in $S_n$, and $\pi \lor \pi'$ is the common coarsening (or \emph{join}) of $\pi$ and $\pi'$ in the partition lattice $\Pi_n$. 

\begin{example}
    Let $(\sigma, \pi) = (1 \ 2 \ 3)  \mid  (4 \ 5)  \mid  (6)(7)$ and $(\sigma', \pi')=(1 \ 4) \mid (2 \ 3) \mid (5)  \mid  (6 \ 7)$. Then $(\sigma, \pi) \cdot (\sigma', \pi') = (1 \ 5 \ 4 \ 2)(3) \mid (6 \ 7). $
\end{example}

\begin{lemma}
    The multiplication $\cdot$ is well-defined and endows $\allpp_n$ with the structure of a monoid.
\end{lemma}

\begin{proof}
Let $(\sigma, \pi), (\sigma', \pi') \in \allpp_n$. For any $i \in [n]$, $i$ and $\sigma'(i)$ lie in the same block of $\pi'$, while $\sigma'(i)$ and $\sigma(\sigma'(i)) = (\sigma \cdot \sigma')(i)$ lie in the same block of $\pi$. Therefore $i$ and $(\sigma \cdot \sigma')(i)$ lie in the same block of the common coarsening $\pi \lor \pi'$, so the product is well-defined.

Since multiplication in $S_n$ and join in $\Pi_n$ are associative, and $(1) \mid \dots \mid (n)$ acts as the identity element, $(\allpp_n, \cdot)$ forms a monoid.
\end{proof}

By replacing $[n]$ with any finite set $C$, we can likewise define partitioned permutations of $C$, which we denote by $\allpp_C$. %(We will use $\hat{1}_C$ to denote the maximum element of $\Pi_C$, the lattice of set partitions of $C$.)
If $C \subseteq D$, then by adding to any $(\sigma,\pi) \in \allpp_C$ the elements of $D \setminus C$ as $1$-cycles in singleton blocks, we can embed $\allpp_C$ into $\allpp_D$. In particular, for two sets $C$ and $C'$, this allows us to define the product of elements from $\allpp_C$ and $\allpp_{C'}$ as an element of $\allpp_{C \cup C'}$.

\begin{example}\label{Ex:A_S1_and_A_S2}
    Let 
    \begin{align*}
        \partperm{\sigma}{\pi} &= (1)  \mid  (2 \ 3)(4) \mid (5 \ 6) \in \allpp_{C},\\
        \partperm{\sigma'}{\pi'} &= (1 \ 4) \mid (3)  \mid  (5 \ 7) \mid (8) \in \allpp_{C'},
    \end{align*}
    where $C = \{1,2,3,4,5,6\}$ and $C' = \{1,3,4,5,7,8\}$.
    Then their product is \[\partperm{\sigma}{\pi} \cdot \partperm{\sigma'}{\pi'} = (1 \ 4)(2 \ 3) \mid (5 \ 7 \ 6) \mid (8)\in \allpp_{C \cup C'}.\]
\end{example}

\subsection{Minimal factorizations}\label{subsec:min_fact}
We now discuss how to extend our notion of factorizations in $S_n$ to factorizations in $\allpp_n$.

\begin{definition}\label{def_transposition}
    A \textit{transposition} in $\allpp_n$ is a partitioned permutation $\partperm{s_{ij}}{\pi_{ij}}$, where $s_{ij} = (i \ j) \in T_n$ and $\pi_{ij} = 1 \mid \dots \mid ij \mid \dots \mid n \in \Pi_n$.
\end{definition}

As an abuse of notation, we will often identify the transpositions in $S_n$ with the transpositions in $\allpp_n$ and denote them both by $s_{ij}$ or $t_a$.

The following lemma gives a description of what can happen when we multiply an element of $\allpp_n$ by a transposition. Let $\cycle(\sigma)$ denote the number of cycles in $\sigma \in S_n$, and let $\parts(\pi)$ denote the number of parts in $\pi \in \Pi_n$.

\begin{lemma}\label{Lem:multiply_transp}
    Let $\partperm{\sigma}{\pi} \in \allpp_n$, and let $s_{ij}$ be a transposition. If $\partperm{\sigma'}{\pi'} = \partperm{\sigma}{\pi} \cdot s_{ij}$, then:
    \begin{enumerate}[label=(\alph*)]
        \item $\cycle(\sigma') = \cycle(\sigma) + 1$ and $\parts(\pi') = \parts(\pi)$ if $i$ and $j$ lie in the same block of $\pi$ and the same cycle of $\sigma$;
        \item $\cycle(\sigma') = \cycle({\sigma})-1$ and $\parts(\pi') = \parts(\pi)$ if $i$ and $j$ lie in the same block of $\pi$ but different cycles of $\sigma$;
        \item $\cycle(\sigma') = \cycle({\sigma})-1$ and $\parts(\pi') = \parts(\pi)-1$ if $i$ and $j$ lie in different blocks of $\pi$ (and hence different cycles of $\sigma$).
    \end{enumerate}
\end{lemma}

\begin{proof}
    If $i$ and $j$ lie in the same cycle of $\sigma$, then multiplying by $s_{ij}$ breaks this cycle into two cycles, so $\cycle(\sigma') = \cycle(\sigma)+1$. However, if $i$ and $j$ lie in different cycles of $\sigma$, then multiplying by $s_{ij}$ merges these two cycles, so $\cycle(\sigma') = \cycle(\sigma)-1$.

    If $i$ and $j$ lie in the same block of $\pi$, then $\pi' = \pi \lor \pi_{ij} = \pi$, so $\parts(\pi') = \parts(\pi)$. Otherwise, $i$ and $j$ lie in different blocks, so $\pi'$ is formed from $\pi$ by merging these two blocks and thus $\parts(\pi') = \parts(\pi)-1$.
\end{proof}

Given any factorization $f$ of $\sigma \in S_n$, the product of the corresponding sequence of transpositions in $\allpp_n$ is the element $(\sigma, \pi) \in \allpp_n$, where the blocks of $\pi$ give the connected components of $\graphf_f$; we call this sequence a factorization of $(\sigma, \pi)$.

We are particularly interested in factorizations of $\partperm{\sigma}{\pi}$ of minimal length (or \emph{minimal factorizations}). Let $\len \partperm{\sigma}{\pi}$ denote the minimal length of a factorization $f$ of $\partperm{\sigma}{\pi}$. Then we have the following lemma.

\begin{lemma}\label{Lem:minlen_pp}
    Let $\partperm{\sigma}{\pi} \in \allpp_n$. Then
    \begin{equation*}\label{eq:min_len_pp}
        \len \partperm{\sigma}{\pi} = n+\cycle(\sigma)-2\parts(\pi).
    \end{equation*}
    Furthermore, if $f$ is a minimal factorization of $\partperm{\sigma}{\pi}$, then when we multiply transpositions from left to right in $f$, only cases (a) and (c) in Lemma~\ref{Lem:multiply_transp} occur.     
\end{lemma}

\begin{proof} 
Suppose $f = (t_1, \dots, t_k)$ is a factorization of $\partperm{\sigma}{\pi}$, and let $\partperm{\sigma_i}{\pi_i} = t_1\cdots t_i$ with $1 \le i \le k$. Let 
\[l_i =n + \cycle(\sigma_i) - 2 \parts(\pi_i). \]
Since $(\sigma_0, \pi_0) = (1) \mid (2) \mid \dots \mid (n)$, we have $l_0 = n+n-2n = 0$.

If multiplying $(\sigma_i, \pi_i)$ by $t_{i+1}$ is as in case (a) or (c) of Lemma~\ref{Lem:multiply_transp}, then we have $l_{i+1} = l_{i}+1$; if it is as in case (b), then we have $l_{i+1} = l_{i}-1$. Thus $l_i$ can increase by at most $1$ at each step, so any factorization of $\partperm{\sigma}{\pi}$ requires at least $l_k = n+\cycle(\sigma) - 2\parts(\pi)$ transpositions. We will show that a factorization of this length exists, from which the result will follow.

First, for each block $B_i$ of $\pi$, we multiply $|B_i|-1$ transpositions to form a single cycle that is a concatenation of all the cycles of $\sigma$ in this block. This requires $\sum_i (|B_i|-1) = n - \parts(\pi)$ transpositions to form all $\parts(\pi)$ cycles. Next, in each block, we multiply by the appropriate transpositions to break this cycle into the cycles of $\sigma$. This requires $\cycle(\sigma) - \parts(\pi)$ additional transpositions for a total of $n+ \cycle(\sigma) - 2\parts(\pi)$ transpositions.
\end{proof}

\begin{example}\label{ex:construct_min_pp}
    We explicitly construct a minimal factorization of $(\sigma, \pi) = (1 \ 2 \ 3) (6 \ 7)  \mid  (4 \ 5)  \mid  (8) \in \allpp_8$ according to the proof of Lemma~\ref{Lem:minlen_pp}. 
    We first build a full cycle on each block of $\pi$ such as $(1 \ 2 \ 3 \ 6 \ 7)\mid(4 \ 5)\mid(8)$ using 5 transpositions: $(s_{17}, s_{16},s_{13},s_{12},s_{45})$. Since $(1 \ 2 \ 3 \ 6 \ 7)s_{37} = (1 \ 2 \ 3)(6 \ 7),$ it follows that $(s_{17}, s_{16},s_{13},s_{12},s_{45},s_{37})$ is a minimal factorization of $\partperm{\sigma}{\pi}$.
\end{example}

From Lemma~\ref{Lem:minlen_pp}, we obtain the length of a minimal transitive factorization of a permutation as a special case; see also \cite{Goulden-Jackson} for an alternative proof. 

\begin{corollary} \label{cor:minlength}
    The length of a minimal transitive factorization of $\sigma \in S_n$ is $\len(\sigma) = n+\cycle(\sigma)-2$. In particular, if $\sigma$ is the identity, then $\len(\id) = 2n-2$.
\end{corollary}

\begin{proof}
    Transitive factorizations of $\sigma \in S_n$ are just factorizations of $\partperm{\sigma}{\hat{1}} \in \allpp_n$. Since $\parts(\hat{1})=1$, we obtain $\len\partperm{\sigma}{\hat{1}} = n+\cycle(\sigma)-2$ from Lemma~\ref{Lem:minlen_pp}.
\end{proof}

\subsection{Quasi-threshold graphs}\label{subsec:quasi_threshold_graph}
%some more introductions
We will mainly study factorizations supported on a class of graphs called quasi-threshold graphs, also known as trivially perfect graphs. See, for instance, \cite{quasi_threshold_graph} for more background on quasi-threshold graphs, including alternative characterizations and special properties.

\begin{definition}
    Given a poset $P$, its \emph{comparability graph} $G(P)$ is the graph on vertex set $P$ with an edge connecting $p, q \in P$ if and only if $p$ and $q$ are comparable in $P$.

    A \emph{rooted tree} is a poset whose Hasse diagram is a tree with a unique maximal element called the \emph{root}. A \emph{rooted forest} is a disjoint union of rooted trees. 
    
    A graph $G$ is a \textit{quasi-threshold graph} if it is the comparability graph of a rooted forest.
\end{definition}

From this definition, one can see that any induced subgraph of a quasi-threshold graph $G$ is again quasi-threshold, as any induced subposet of a rooted forest is also a rooted forest.

\begin{remark}
    Quasi-threshold graphs can also be constructed recursively from single-vertex graphs using the operations of (a) disjoint union and (b) adding a new vertex adjacent to all existing vertices \cite{quasi_threshold_graph}.
\end{remark}

We will assume throughout that all quasi-threshold graphs are obtained from rooted forests that are \textit{naturally labeled}, meaning that the partial order on the vertices is compatible with the usual ordering on $[n]$; see Figure~\ref{fig:graph_to_tree} for an example. This assumption can be made without loss of generality as it will not affect the factorization counts after relabeling appropriately.

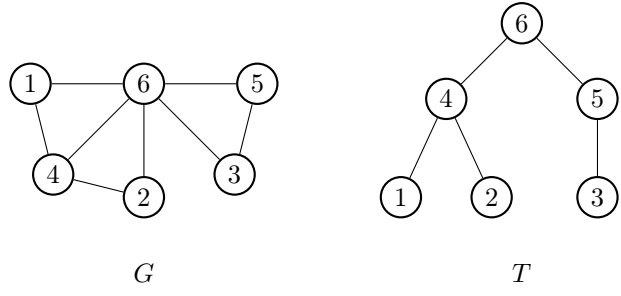
\begin{figure}
  \centering
  \begin{tikzpicture}[
node/.style={circle, draw=black, thick, minimum size=4mm, inner sep=2.5pt},
]
%Nodes
\node[node] at (0,0) (1) {6};
\node[node] at (1.5,0) (2) {5};
\node[node] at (1.2,-1.2) (4) {3};
\node[node] at (0,-1.5) (5) {2};
\node[node] at (-1.2, -1.2) (3) {4};
\node[node] at (-1.5,0) (6) {1};

%Lines
\draw[-] (1) -- (2);
\draw[-] (1) -- (3);
\draw[-] (1) -- (4);
\draw[-] (1) -- (5);
\draw[-] (1) -- (6);
\draw[-] (4) -- (2);
\draw[-] (3) -- (5);
\draw[-] (3) -- (6);

\node at (0,-2.5) {$G$};

\begin{scope}[shift={(5,-1.5)}]
%Nodes
\node[node] at (0,2.3) (1) {6};
\node[node] at (1,1.3) (2) {5};
\node[node] at (-1, 1.3) (3) {4};
\node[node] at (1,0) (4) {3};
\node[node] at (-0.4,0) (5) {2};
\node[node] at (-1.6,0) (6) {1};

%Lines
\draw[-] (1) -- (2);
\draw[-] (1) -- (3);
\draw[-] (2) -- (4);
\draw[-] (3) -- (5);
\draw[-] (3) -- (6);
\node at (0,-1) {$T$};
\end{scope}
\end{tikzpicture}
\caption{A quasi-threshold graph $G$, obtained as the comparability graph of the (naturally labeled) rooted tree $T$.}
\label{fig:graph_to_tree}
\end{figure}

\subsection{Hurwitz action in quasi-threshold graphs} \label{subsec:Hurwitz}

The Hurwitz action, first introduced in \cite{Hurwitz_ori_paper}, is an action on factorizations (of a general group) that has been well-studied in finite Coxeter groups (see \cite{fin_coxeter_Hwz_action}) and, more generally, in complex reflection groups (see \cite{complex_ref_gps_Hwz_action}). It is a common tool for counting transitive factorizations; see \cite{symm_gp_Hwz_equiv, coxeter_count_reflections} for more details. 

\begin{definition}
    Let $f = (t_1, t_2, \dots, t_k)$  be a factorization of $\sigma \in S_n$. The \emph{Hurwitz move} $R_i$ acts on $f$ by 
    \[R_if = R_i(t_1, \dots, t_i,t_{i+1},\dots, t_k) = (t_1, \dots,t_{i+1},t_{i+1}^{-1}t_it_{i+1}, \dots, t_k).\]
    The \emph{inverse Hurwitz move} $R_i^{-1}$ is defined similarly:
    \[R_i^{-1}(t_1, \dots, t_i,t_{i+1},\dots, t_k) = (t_1, \dots,t_it_{i+1}t_i^{-1}, t_i, \dots, t_k).\]
\end{definition}

Observe that if $f$ is a factorization of $(\sigma, \pi)$, then so are $R_i f$ and $R_i^{-1} f$. While Hurwitz moves can be defined for general groups, note that in our case we have $t_i=t_i^{-1}$ for all $i$ since $t_i$ is a transposition.

The following lemma follows from a straightforward calculation.
\begin{lemma}\label{lem:braid_relations}
    Hurwitz moves satisfy the braid relations:
    \begin{enumerate}[label=(\alph*)]
        \item $R_iR_{i+1}R_i = R_{i+1}R_iR_{i+1}$ for all $i$; 
        \item $R_iR_j = R_jR_i$ for $|i-j| \ge 2.$ 
    \end{enumerate}
\end{lemma}

Because of this, we can define a Hurwitz operator for any permutation $w \in S_k$ by finding a reduced expression for $w$ as a minimal product of simple transpositions, say $w = s_{i_1}\cdots s_{i_p}$. We then define $R_w = R_{i_1}\cdots R_{i_p}$.  Since Hurwitz moves satisfy the braid relations, the action of the resulting $R_w$ does not depend on the choice of reduced expression for $w$ by Matsumoto's theorem.

We next prove a lemma that describes how certain Hurwitz moves affect factorizations supported on a quasi-threshold graph. It will be useful to first define the following notions of rank.

\begin{definition}
    The \emph{rank} of a transposition $s_{ij}$ is $\rank(s_{ij}) = \max\{i,j\}$.

    The \emph{rank sequence} of a factorization $f = (t_1, \dots, t_k)$ is $\rank(f) = (\rank(t_1), \dots, \rank(t_k))$. 
\end{definition}

\begin{lemma}\label{Lem:one_Hwz_move}
    Let $f = (t_1, t_2, \dots, t_k)$ and $f' = (t_1', t_2', \dots, t_k')$ be factorizations such that $f' = R_i f$ for some $i$. If $\rank(t_i) > \rank(t_{i+1})$ or $\rank(t'_i) < \rank(t'_{i+1})$, then
    \begin{enumerate}[label=(\alph*)]
        \item $\rank(t'_i) = \rank(t_{i+1})$ and $\rank(t'_{i+1}) = \rank(t_i)$; and
        \item for any quasi-threshold graph $G$, $f$ is supported on $G$ if and only if $f'$ is.
    \end{enumerate}
\end{lemma}

\begin{proof} 
    If $t_i$ and $t_{i+1}$ commute, then $t'_i = t_{i+1}$ and $t'_{i+1} = t_i$, so both (a) and (b) are clear. Otherwise, suppose $t_i = (a \ b)$ and $t_{i+1} = t_i' = (b \ c)$, so that $t_{i+1}' = t_{i+1}t_it_{i+1} = (a \ c)$. If $\rank(t_i) > \rank(t_{i+1})$, then $\max\{a,b\} > \max\{b,c\}$, so we must have $a > b$ and $a > c$; we get the same inequalities if instead $\rank(t_i') < \rank(t'_{i+1})$. Hence $\rank(t_i) = a = \max(a,c) = \rank(t_{i+1}t_it_{i+1})$, proving (a).

    For (b), let $G$ be the comparability graph of the (naturally labeled) rooted forest $F = (F, \prec)$. If $f$ is supported on $G$, then $a$ and $b$ are comparable in $F$, as are $b$ and $c$. We must then have either the chain $a\succ b \succ c$ or $a \succ c \succ b$ in $F$. (If $a \succ b$ and $c \succ b$, then $a$ and $c$ must be comparable since $F$ is a rooted forest.) Therefore $a$ and $c$ are comparable in $F$, so the edge $\{a,c\}$ exists in $G$, implying $f'$ is supported on $G$. The converse direction is similar.
\end{proof}

One way of interpreting Lemma~\ref{Lem:one_Hwz_move}(a) is that such Hurwitz moves perform a simple transposition on the rank sequence of $f$. We will therefore be able to use Hurwitz moves to sort a factorization by rank.

\begin{definition}

    Let $f= (t_1, \dots, t_k)$ be a factorization with rank sequence $r = (r_1, \dots, r_k)$. The \emph{standardizing permutation} $w^{(f)} = w_1 \dots w_k \in S_k$ (in one-line notation) is the unique permutation such that if $w_i < w_j$, then $r_i\le r_j$; if furthermore $r_i = r_j$, then $i<j$. 

    We say $f$ is in \emph{standard form} if its rank sequence is weakly increasing (or equivalently, if $w^{(f)} = \id$).
\end{definition}

In other words, the one-line notation of $w^{(f)}$ labels the positions of the rank sequence in weakly increasing order, breaking ties from left to right.
    
\begin{example}\label{ex:permute_rank} The factorization $\Tilde{f} = (s_{14}, s_{24},s_{35},s_{26},s_{56})$ has rank sequence $\rank(\Tilde{f}) = (4,4,5,6,6)$ and is thus in standard form. 

    Consider the factorization $f= (s_{16},s_{35},s_{14},s_{56},s_{24})$. Then $\rank(f) = (6,5,4,6,4)$ and $w^{(f)} = 43152$. The standardizing permutation $w^{(f)}$ permutes the entries of $\rank(f)$ to $(4,4,5,6,6)$.

    Note that both these factorizations are supported on the graph $G$ shown in Figure~\ref{fig:graph_to_tree}. 
\end{example}

As in our discussion after Lemma~\ref{lem:braid_relations}, one can use a reduced expression for $w^{(f)}$ to obtain a sequence of Hurwitz moves $R_{w^{(f)}}$ that standardizes $f$.

\begin{lemma}\label{lem:std_form}
    Suppose $f$ is a factorization with standardizing permutation $w^{(f)}$. Then the factorization $\Tilde{f} = R_{w^{(f)}} f$ is in standard form. Moreover, for any quasi-threshold graph $G$, $f$ is supported on $G$ if and only if $\tilde{f}$ is.

    Moreover, for any factorization $\Tilde{f}$ in standard form, there is a bijection between distinct permutations of the sequence $\rank(\Tilde{f})$ and factorizations $f$ such that $R_{w^{(f)}} f = \Tilde{f}$.
\end{lemma}

\begin{proof}
    Let us choose a reduced expression for $w^{(f)} = s_{i_1}\cdots s_{i_p}$ so that applying $R_{w^{(f)}} = R_{i_1}\cdots R_{i_p}$ to $f$ first moves the elements of lowest rank $r$ to the left (past elements of strictly higher rank as per Lemma~\ref{Lem:one_Hwz_move}), then elements of rank $r+1$, and so forth. The result is a factorization $\Tilde f = R_{w^{(f)}} f$ in standard form whose rank sequence is that of $f$ arranged in weakly increasing order. By Lemma~\ref{Lem:one_Hwz_move}, $\Tilde{f}$ is supported on a quasi-threshold graph $G$ if and only if $f$ is supported on $G$.
    
    Conversely, for any permutation of $\rank(\Tilde f)$, any factorization with this permuted rank sequence has the same standardizing permutation $w$. Then the unique factorization $f$ with this permuted rank sequence standardizing to $\tilde{f}$ is $f = R_w^{-1} \tilde{f}$ (which by Lemma~\ref{Lem:one_Hwz_move} will only involve moving elements to the right past ones of strictly higher rank).
\end{proof}

It follows immediately from Lemma~\ref{lem:std_form} that if $\Tilde{f}$ is in standard form with $k_i$ transpositions of rank $i$ (with $k=\sum_{i=1}^r k_i$), then there are $\binom{k}{k_1, \dots, k_r}$ factorizations $f$ such that $R_{w^{(f)}} f = \Tilde{f}$. 

%We now give an example of obtaining $R_{w^{(f)}} f$ from $f$ using the standardizing permutation $w^{(f)}$. 

\begin{example}
    Let $f= (s_{16},s_{35},s_{14},s_{56},s_{24})$ as in Example~\ref{ex:permute_rank}. Then $w^{(f)} = 43152 = s_3s_2s_3s_4s_1s_2$. Thus we can compute $\Tilde{f}$ by applying the corresponding Hurwitz moves (from right to left) as follows:
    \[
    \begin{split}
        f = (s_{16},s_{35},s_{14},s_{56},s_{24}) &\xrightarrow{R_2}\\
         (s_{16},s_{14},s_{35},s_{56},s_{24}) &\xrightarrow{R_1}\\
         (s_{14},s_{46},s_{35},s_{56},s_{24}) &\xrightarrow{R_4}\\
         (s_{14},s_{46},s_{35},s_{24},s_{56}) &\xrightarrow{R_3}\\
         (s_{14},s_{46},s_{24},s_{35},s_{56}) &\xrightarrow{R_2}\\
         (s_{14},s_{24},s_{26},s_{35},s_{56}) &\xrightarrow{R_3}\\
         (s_{14},s_{24},s_{35},s_{26},s_{56}) &= \Tilde{f}.
    \end{split}
    \]
\end{example}

By considering standardized factorizations, we will be able to isolate the parts of each factorization with a given rank. In the next section, we will introduce the key tool for counting these parts of fixed rank.

\section{Lattice walks and factorizations}\label{sec:lattice_walk}

In this section, we will describe an important lemma on lattice walks based on the well-known \emph{cycle lemma} of Dvoretzky and Motzkin \cite{paper:cycle_lemma_Dvoretzky_Motzkin}; see also \cite{paper:cycle_lemma} for further applications. We will then show how this lemma can be used to count factorizations of a certain special form.

\subsection{Lattice walks and the cycle lemma}

To describe the variation of the cycle lemma we need, we first define a particular type of lattice walk.

\begin{definition}
    A \emph{lattice walk} of length $k$ is a finite sequence $\lattice = (l_1, \ldots, l_k)$ of nonzero integers (called \emph{steps}) summing to $-1$. We say that the $i$th step is an \emph{up step} if $l_i > 0$ and a \emph{down step} if $l_i < 0$.

    The \emph{type} of $\lattice$ is the pair of integer partitions $(\lambda, \mu)$, where the parts of $\lambda$ and $\mu$ are the magnitudes of the up and down steps of $\lattice$, respectively.
\end{definition}

\begin{definition}
    For a lattice walk $\lattice = (l_1, \dots, l_k)$, its \emph{heights} $(h_0, h_1, \dots, h_k)$ are defined by $h_0 = 1$ and $h_i = h_{i-1} + l_i$ for $i= 1, \dots, k$.

    We say that $\lattice$ is \emph{good} if $h_i > 0$ for all $0\le i\le k-1$ (and $h_k = 0$); otherwise, we say that $\mathscr{L}$ is \emph{bad}.
\end{definition}

We will visualize a lattice walk by drawing its \emph{graph} formed by connecting the points $(i, h_i)$ for $i=0, 1, \dots, k$. Then $\lattice$ is good if its graph remains strictly above the $x$-axis before ending at $(k,0)$. 

\begin{example} \label{ex:latticewalks}
    Figure~\ref{fig:lattice_walk2} gives an example of a good lattice walk $\lattice_1 = (3,-2,2,-1,-2,1,-2)$ and a bad lattice walk $\lattice_2 = (2,-1,-2,1,-2,3,-2)$. Both walks have type $(321,2221)$.
\end{example}

\begin{figure}
\begin{tikzpicture}[scale=.9]

  \draw[thick, ->] (0,0) -- (7.5,0);
  \draw[thick, ->] (0,0) -- (0,4.5);
  \foreach \x in {0,1,2,3,4,5,6,7}
  \draw (\x cm,1pt) -- (\x cm,-1pt) node[anchor=north] {$\x$};
  \foreach \y in {0,1,2,3,4}
  \draw (1pt,\y cm) -- (-1pt,\y cm) node[anchor=east] {$\y$};

  \draw[thick, ->] (0,0) -- (7.5,0) node[anchor=north west] {$i$};
  \draw[thick, ->] (0,0) -- (0,4.5) node[anchor=south east] {$h_i$};

  \draw[thick, red, -] (0,1) -- (1,4) node[above left, pos=0.5] {\textcolor{red}{3}};
  \draw[thick, blue, -] (1,4) -- (2,2) node[above right, pos=0.5] {\textcolor{blue}{2}};
  \draw[thick, red, -] (2,2) -- (3,4) node[above left, pos=0.5] {\textcolor{red}{2}};
  \draw[thick, blue, -] (3,4) -- (4,3) node[above right, pos=0.5]  {\textcolor{blue}{1}};
  \draw[thick, blue, -] (4,3) -- (5,1) node[above right, pos=0.5] {\textcolor{blue}{2}};
  \draw[thick, red, -] (5,1) -- (6,2) node[above left, pos=0.5] {\textcolor{red}{1}};
  \draw[thick, blue, -] (6,2) -- (7,0) node[above right, pos=0.5] {\textcolor{blue}{2}};

  \fill (0,1)  circle (2pt);
  \fill (1,4)  circle (2pt);
  \fill (2,2)  circle (2pt);
  \fill (3,4)  circle (2pt);
  \fill (4,3)  circle (2pt);
  \fill (5,1)  circle (2pt);
  \fill (6,2)  circle (2pt);
  \fill (7,0)  circle (2pt);
  \fill[white] (5,-0.9)  circle (2pt);
  \node at (3.5,-1){$\mathscr{L}_1$};

  \begin{scope}[shift={(9,0)}]
    \draw[thick, ->] (0,0) -- (7.5,0);
    \draw[thick, ->] (0,0) -- (0,4.5);
    \foreach \x in {0,1,2,3,4,5,6,7}
    \draw (\x cm,1pt) -- (\x cm,-1pt) node[anchor=north] {$\x$};
    \foreach \y in {0,1,2,3,4}
    \draw (1pt,\y cm) -- (-1pt,\y cm) node[anchor=east] {$\y$};

    \draw[thick, ->] (0,0) -- (7.5,0) node[anchor=north west] {$i$};
    \draw[thick, ->] (0,0) -- (0,4.5) node[anchor=south east] {$h_i$};

    \draw[thick, red, -] (0,1) -- (1,3) node[above left, pos=0.58] {\textcolor{red}{2}};
    \draw[thick, blue, -] (1,3) -- (2,2) node[above right, pos=0.5] {\textcolor{blue}{1}};
    \draw[thick, blue, -] (2,2) -- (3,0) node[above right, pos=0.5]  {\textcolor{blue}{2}};
    \draw[thick, red, -] (3,0) -- (4,1) node[above left, pos=0.5]  {\textcolor{red}{1}};
    \draw[thick, blue, -] (4,1) -- (5,-1) node[above right, pos=0.5]  {\textcolor{blue}{2}};
    \draw[thick, red, -] (5,-1) -- (6,2) node[above left, pos=0.5] {\textcolor{red}{3}};
    \draw[thick, blue, -] (6,2) -- (7,0) node[above right, pos=0.5]  {\textcolor{blue}{2}};
    \fill (0,1)  circle (2pt);
    \fill (1,3)  circle (2pt);
    \fill (2,2)  circle (2pt);
    \fill (3,0)  circle (2pt);
    \fill (4,1)  circle (2pt);
    \fill (5,-1)  circle (2pt);
    \fill (6,2)  circle (2pt);
    \fill (7,0)  circle (2pt);
  \node at (3.5,-1){$\mathscr{L}_2$};
  \end{scope}
\end{tikzpicture}

\caption{A good lattice walk $\lattice_1$ and a bad lattice walk $\lattice_2$.}
\label{fig:lattice_walk2}

\end{figure}
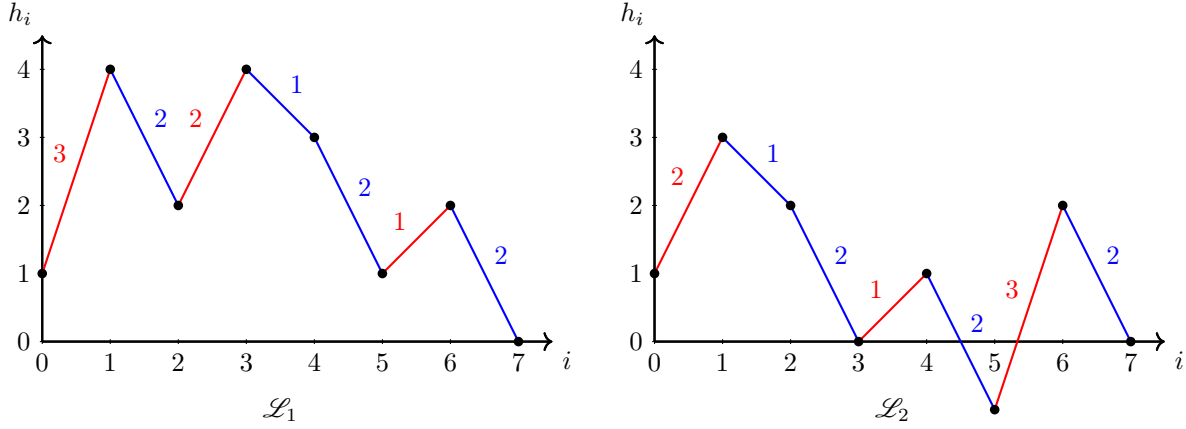

Let $\phi$ denote cyclic rotation of a lattice walk, so that if $\lattice = (l_1, \dots, l_k)$, then $\phi(\lattice) = (l_2, \dots, l_k, l_1)$. Denote the orbit of $\phi$ acting on $\lattice$ by $\Phi(\lattice) = \{\phi^m(\lattice) \mid 1 \leq m \leq k\}$. Note that since $\lattice$ sums to $-1$, it cannot be periodic, so $|\Phi(\lattice)|=k$. We then have the following version of the cycle lemma.

\begin{lemma}\label{lem:cycle_lem}
    For any lattice walk $\lattice$, the orbit $\Phi(\lattice)$ contains a unique good lattice walk. 
\end{lemma}

\begin{proof}
    Let $(h_0, \dots, h_k)$ be the heights of $\lattice$. For fixed $m$ with $1 \leq m \leq k$, let $\phi^m(\lattice)$ have heights $(h'_0, \dots, h'_k)$. Then 
    \[\label{eq:good_walk}
    h'_j = \begin{cases}
        h_{j+m}-h_m+1 & \text{for }  j \leq k-m,\\
        {h_{j-k+m} - h_m} & \text{for } j \geq k-m.\\
        \end{cases}
    \]

    We claim that there is exactly one $m$ for which $h'_j > 0 $ for all $0 \leq j < k$. This condition is equivalent to $h_m \leq h_{j+m}$ for $0 \leq j \leq k-m$ and $h_m < h_{j-k+m}$ for $k-m \leq j < k$, or after reindexing, $h_m \leq h_{m'}$ for $m \leq m' \leq k$ and $h_m < h_{m'}$ for $0 \leq m' < m$. The unique value of $m$ for which these are satisfied is the first index of the lowest height appearing in $\lattice$.
\end{proof}

\begin{example}
    Consider again the lattice walks in Example~\ref{ex:latticewalks} as shown in Figure~\ref{fig:lattice_walk2}. The lowest height of the bad lattice walk $\lattice_2$ is $h_5 = -1$. Hence $\phi^5(\lattice_2) = \lattice_1$ is the unique good lattice walk in its orbit. 
\end{example}

From Lemma~\ref{lem:cycle_lem}, we can count the number of good lattice walks by type. For an integer partition $\lambda = (\lambda_1, \dots, \lambda_k)$ with $\lambda_1 \geq \cdots \geq \lambda_k > 0$, let
\begin{itemize}
    \item $|\lambda| = \lambda_1 + \cdots + \lambda_k$, the sum of the parts of $\lambda$;
    \item $\ell(\lambda) = k$, the number of parts of $\lambda$;
    \item $\Pi(\lambda) = \prod_{i=1}^k \lambda_i$, the product of the parts of $\lambda$; and
    \item $\aut(\lambda) = \prod_i m_i!$, where $m_i$ is the number of parts of $\lambda$ of size $i$.
\end{itemize}
\begin{corollary} \label{cor:goodwalks}
    Let $\lambda$ and $\mu$ be partitions with $|\mu| = |\lambda|+1$. Then the number of good lattice walks of type $(\lambda, \mu)$ is \[\frac{(\ell(\lambda) + \ell(\mu)-1)!}{\aut(\lambda) \aut(\mu)}.\]
\end{corollary}
\begin{proof}
The total number of lattice walks of type $(\lambda, \mu)$ is the multinomial coefficient $\frac{(\ell(\lambda)+\ell(\mu))!}{\aut(\lambda)\aut(\mu)}$. By Lemma~\ref{lem:cycle_lem}, there is exactly one good lattice walk in each cyclic orbit of size $\ell(\lambda)+\ell(\mu)$, so dividing by this gives the result.
\end{proof}

\subsection{Factorizations supported on stars}

We now relate good lattice walks to certain types of factorizations supported on stars. For a permutation $\sigma$, let $\type(\sigma)$ be the integer partition denoting its cycle type, and for a set partition $\pi$, let $\type(\pi)$ be the integer partition denoting the sizes of its blocks.

\begin{theorem} \label{thm:lattice_walk_count}
Let $H$ be a star graph on vertex set $V \subseteq [n]$ with center vertex $z \in V$, and let $\lambda$ and $\mu$ be partitions with $|\mu| = |\lambda| + 1 = |V| = m$. Fix a partitioned permutation $(\sigma, \pi) \in \allpp_{V \setminus \{z\}}$ with $\type(\sigma) = \type(\pi) = \lambda$. Then
\begin{enumerate}[label = (\alph*)]
    \item the minimum length of a factorization $f=(t_1, \dots, t_k)$ supported on $H$ such that $(\sigma', \pi') = (\sigma,\pi) \cdot t_1 \cdots t_k \in \allpp_V$ satisfies $\type(\sigma') = \mu$ and $\pi' = \hat{1}_V$ is $\ell(\lambda) + \ell(\mu) - 1$, and
    \item the number of such minimal factorizations $f$ is \[\frac{(\ell(\lambda) + \ell(\mu)-1)!}{\aut(\mu)} \cdot \Pi(\lambda).\] 
\end{enumerate}
\end{theorem}

\begin{proof}
    Suppose that we have such an $f = (t_1, \dots, t_k)$. We associate to $f$ a good lattice walk $\lattice$ with heights $(h_0, \dots, h_k, h_{k+1})$ such that, for $0 \leq i \leq k$, $h_i$ is the length of the cycle containing $z$ in $\partperm{\sigma}{\pi} \cdot t_1 \dots t_i$. Note that $h_0=1$ since the embedding of $\partperm{\sigma}{\pi}$ into $\allpp_V$ has $z$ in a $1$-cycle.
    
    By Lemma~\ref{Lem:minlen_pp}, 
    \begin{align*}
        \len\partperm{\sigma}{\pi} &= |V \setminus \{z\}| + \cycle(\sigma) - 2\parts(\pi) = m-1 - \ell(\lambda),\\
        \len\partperm{\sigma'}{\pi'} &= |V| + \cycle(\sigma') - 2\parts(\pi') = m+\ell(\mu) - 2.
    \end{align*} Therefore, the length of $f$ must be at least
    \[
        \len\partperm{\sigma'}{\pi'} - \len\partperm{\sigma}{\pi} = \ell(\lambda)+\ell(\mu)-1.
    \]
    To achieve this length, multiplying by each transposition $s_{iz}$ in $f$ needs to be as in case (a) or (c) of Lemma~\ref{Lem:multiply_transp}, i.e., either (a) $i$ and $z$ are already in the same cycle, or (c) they are in different cycles and different parts.

    In case (a), we split the cycle containing $i$ and $z$: $(\dots  i \dots  z )s_{iz} = (\dots i )(\dots z)$. By Lemma~\ref{Lem:multiply_transp}, we cannot further modify the new cycle containing $i$ again (as this would result in case (b)), so it must end up as a cycle in $\sigma'$. If this cycle has length $l$, then the length of the cycle containing $z$ decreases by $l$, so we have a down step of size $-l$ in $\lattice$. Note that as long as the cycle containing $z$ has size bigger than $l$, there is a unique choice of transposition $s_{iz}$ that will result in this down step.

    In case (c), we merge the cycle containing $z$ with one of the original cycles of $\sigma$: $(\dots i )(\dots z)s_{iz} = (\dots  i \dots  z )$. This increases the length of the cycle containing $z$ by the length $l$ of the merged cycle, so it results in an up step of size $l$ in $\lattice$. For a given cycle of $\sigma$ of length $l$, there are $l$ choices for the transposition $s_{iz}$ that will result in this up step. Since $\pi' = \hat{1}_{V}$, each cycle of $\sigma$ must be merged exactly once.

    It follows that each up step of $\lattice$ corresponds to a cycle of $\sigma$ together with a choice of element in that cycle, while each down step of $\lattice$ corresponds to a cycle of $\sigma'$ of the same length, which is uniquely determined. (The final down step corresponds to the cycle of $\sigma'$ containing $z$.) Therefore, we can choose $f$ by:
    \begin{itemize}
        \item choosing a good lattice walk $\lattice$ of type $(\lambda, \mu)$;
        \item assigning each cycle of $\sigma$ to an up step of $\lattice$ of the same size; and
        \item choosing an element in each cycle of $\sigma$.
    \end{itemize}

    Using Corollary~\ref{cor:goodwalks} for the first choice, it follows that the total number of such factorizations is
    \[ \frac{(\ell(\lambda) + \ell(\mu)-1)!}{\aut(\lambda)\aut(\mu)} \cdot \aut(\lambda) \cdot \Pi(\lambda) = \frac{(\ell(\lambda) + \ell(\mu)-1)!}{\aut(\mu)} \cdot \Pi(\lambda),\]
    as desired.
    \end{proof}

\begin{example}\label{ex:construct_lattice}
    Let $(\sigma, \pi) = (1) \mid (2\ 3) \mid (4\ 5\ 6) \in \allpp_{6}$, so that $\lambda = \type(\sigma) = \type(\pi) = 321$, and let $\mu = 2221$. Then consider multiplying $(\sigma,\pi)$ (after embedding into $\allpp_7$) by the $\ell(\lambda) + \ell(\mu)-1 = 6$ transpositions $s_{67}s_{57}s_{27}s_{37}s_{67}s_{17}$:
    \begin{align*}
        (\sigma, \pi) = (1) \mid (2\ 3) \mid (4\ 5\ 6) \mid (7)&\xrightarrow{\cdot s_{67}} \\
        (1) \mid (2\ 3) \mid (4\ 5\ 6\ 7) &\xrightarrow{\cdot s_{57}}\\
        (1) \mid (2\ 3) \mid (4\ 5)(6\ 7) &\xrightarrow{\cdot s_{27}}\\
        (1) \mid (4\ 5)(3\ 2\ 6\ 7) &\xrightarrow{\cdot s_{37}}\\
        (1) \mid (4\ 5)(3)(2\ 6\ 7) &\xrightarrow{\cdot s_{67}}\\
        (1) \mid (4\ 5)(3)(2\ 6)(7) &\xrightarrow{\cdot s_{17}}\\
        (4\ 5)(3)(2\ 6)(1\ 7) &= (\sigma',\pi').
    \end{align*}

    The resulting partitioned permutation $(\sigma',\pi')$ has one part and $\type(\sigma') = \mu = 2221$. Recording the lengths of the cycle containing $7$ (and appending $0$ to the end) gives $(1, 4, 2, 4, 3, 1, 2, 0)$, which are the heights of the good lattice walk $\lattice_1$ of type $(\lambda, \mu)$ shown in Figure~\ref{fig:lattice_walk2}. (This sequence of transpositions is one of $\aut(\lambda) \cdot \Pi(\lambda) = 6$ that will result in this lattice walk since we had a choice of how to merge each of the cycles of $\sigma$ into the cycle containing $7$.)
\end{example}

\section{Factorization trees}\label{sec:factorization_tree}

In this section, we will define a combinatorial object called a \emph{factorization tree} and show how to associate such a tree to any minimal transitive factorization. We will then use factorization trees to give a combinatorial interpretation for the number of minimal transitive factorizations supported on a quasi-threshold graph.

\subsection{Factorization trees}

For a graph $T = ([n], E)$ with edge weights $w\colon E \to \Z_{>0}$, let $\ind(i)$ and $\outd(i)$ denote the sums of the weights along incoming edges and outgoing edges at vertex $i$, respectively, where we orient each edge from the smaller vertex to the larger vertex.

\begin{definition}\label{def:ftree}
    A \emph{factorization tree} (or \emph{forest}) is a weighted tree (or forest) $T =([n],E)$ with edge weights $w\colon E \to \Z_{>0}$ such that $1+\ind(i) \geq \outd(i)$ for all vertices $i$.
\end{definition}

We will write $s(i) = 1+\ind(i) - \outd(i) \in \Z_{\geq 0}$ for the discrepancy at vertex $i$. (We can view the inequality $s(i) \geq 0$ as a sort of \emph{flow constraint}.)

\begin{figure}
    \centering
    \begin{tikzpicture}[
        yscale=-1,
        node/.style={circle, draw=black, thick, minimum size=4mm, inner sep=2.5pt},
        e/.style={fill=white, circle, inner sep=1.5pt, text=red, node font=\footnotesize}
    ]
    %Nodes
    \node[node] at (-1,-0.5) (1) {8};
    \node[node] at (1,-0.5) (2) {7};
    \node[node] at (0,.5) (3) {6};
    \node[node] at (0,2) (4) {5};
    \node[node] at (1,2.8) (5) {4};
    \node[node] at (0,3.2) (6) {3};
    \node[node] at (1,4) (7) {2};
    \node[node] at (-1,2.8) (8) {1};

    \draw[-] (1) to node[e] {2} (3);
    \draw[-] (2) to node[e] {2} (3);
    \draw[-] (3) to node[e] {3} (4);
    \draw[-] (4) to node[e] {2} (5);
    \draw[-] (4) to node[e] {1} (8);
    \draw[-] (5) to node[e] {1} (7);
    \draw[-] (4) to node[e] {1} (6);
    \end{tikzpicture}  
    \caption{An example of a factorization tree.}
    \label{fig:K4_ex0}
\end{figure}
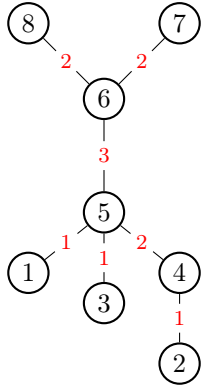

\begin{example}
    An example of a factorization tree is shown in Figure~\ref{fig:K4_ex0}. For example, at vertex $5$, 
    \[1+\ind(5) = 1 + w(4,5) + w(3,5) + w(1,5) = 5 \geq 3 = w(5,6) = \outd(5).\]
    Therefore $s(5) = 5-3=2$. One can likewise check that $s(8) = s(7) = 3$ and $s(i)=0$ for all other vertices $i$.
\end{example}

Factorization trees are also called \emph{genus 0 labeled floor diagrams} in work of Fomin and Mikhalkin \cite{floor_diagrams}. Though it is not important for our purposes, it is shown there that the number of factorization trees on $n$ vertices is $n^{n-2}$. For completeness, we also provide a proof of this (that is very similar to the one in \cite{floor_diagrams}) in Proposition~\ref{prop:cayley}. The $4^2 = 16$ factorization trees on $4$ vertices are shown in Figure~\ref{fig:k4_nid}.

\medskip

The following algorithm assigns a factorization forest to any minimal factorization. During the intermediate stages of the algorithm, each vertex $i$ will carry a label $\sigma^{(i)}$, which will be a permutation (in cycle notation) of a subset of $[n]$. By construction, the domains of the $\sigma^{(i)}$ will be disjoint and partition $[n]$.

\begin{algorithm}\label{alg:labelled} Given a minimal factorization $f$, construct a factorization forest $T$ as follows.

    \begin{enumerate}
    \item Apply Hurwitz moves to $f$ to obtain its standard form $\Tilde{f}$ as in Lemma~\ref{lem:std_form}.
    \item Initialize $T$ on $n$ isolated vertices $\{1,2,\dots,n\},$ and for each vertex $i$, set $\sigma^{(i)} = (i)$. 
    \item For each transposition $s_{ij} \in \tilde{f}$ with $i < j$: 
        \begin{enumerate}
            \item[(a)] Find the unique vertex $k \leq j$ such that $i$ lies in a cycle $\alpha$ of $\sigma^{(k)}$. If $k \neq j$, then:
            \begin{enumerate}
                \item[i.] add the edge $\{j,k\}$ to $T$ with weight $w(k,j) = |\alpha|$, the length of the cycle $\alpha$; %, and label this edge by $\alpha$;
                \item[ii.] move the cycle $\alpha$ from $\sigma^{(k)}$ to $\sigma^{(j)}$.
            \end{enumerate}
            \item[(b)] Multiply $\sigma^{(j)}$ by $s_{ij}$.
        \end{enumerate}
    \end{enumerate}

    The resulting edge-weighted, vertex-labeled graph is the \emph{labeled factorization forest of $f$}. The \emph{(unlabeled) factorization forest of $f$} is obtained by removing the vertex labels.
\end{algorithm}

Note that the process of moving cycles between vertex labels in step (3a) only occurs along edges (and only from smaller vertices to larger vertices). Hence if $i$ lies in a cycle of $\sigma^{(k)}$, then vertices $i$ and $k$ must be in the same connected component.

\begin{example} \label{ex:algstep3}
    Let $\tilde{f} = (s_{45},s_{35},s_{45},s_{36},s_{47},s_{37},s_{47},s_{38},s_{19},s_{29},s_{49},s_{69})$. In Figure~\ref{fig:alg_ex1}, $T_1$ depicts the state of the factorization forest after processing the first $10$ transpositions in $\tilde{f}$ during step (3) of Algorithm~\ref{alg:labelled}.

    When considering $s_{49}$, since $i=4$ lies in the cycle $\alpha = (6\; 4)$ of $\sigma^{(7)}$, we have $k=7$. Therefore we add an edge of weight $2$ between $7$ and $9$, and we move $\alpha$ from $\sigma^{(7)}$ to $\sigma^{(9)}$. We then multiply the new $\sigma^{(9)}$ by $s_{49}$, replacing it with $(6\; 4)(2\; 1\;9) s_{49} = (6\; 4\; 2\; 1\;9)$, resulting in $T_2$.

    When considering $s_{69}$, since $6$ already lies in a cycle of $\sigma^{(9)}$, we simply perform step (3b), replacing $\sigma^{(9)}$ with $(6\; 4\; 2\; 1\;9) s_{69} = (6)(4\; 2\; 1\;9)$, resulting in $T_3$.
\end{example}

\begin{figure}
  \centering
\begin{tikzpicture}[
    yscale=-1,
    node/.style={circle, draw=black, thick, minimum size=4mm, inner sep=2.5pt},
    e/.style={fill=white, circle, inner sep=1.5pt, text=red, node font=\footnotesize}
  ]
  \begin{scope}
    \node[node, label=west:\textcolor{blue}{$(6\;4)(7)$}] at (-0.3,-0.6) (5) {7};
    \node[node, label=north:\textcolor{blue}{$(3 \;8)$}] at (-1,-1.8) (12) {8};
    \node[node] at (-0.3,.8) (6) {6};
    \node[node, label=west:\textcolor{blue}{$(5)$}] at (-1,2) (7) {5};
    \node[node] at (-0.3,3.2) (9) {3};
    \node[node] at (-1.7,3.2) (8) {4};
    \node[node, label=north:\textcolor{blue}{$(2 \;1\; 9)$}] at (1, -1.8) (2) {9};
    \node[node] at (1,-.4) (10) {2};
    \node[node] at (2.3,-.6) (11) {1};

    \draw[-] (12) to node[e] {1} (5);
    \draw[-] (5) to node[e] {3} (6);
    \draw[-] (6) to node[e] {2} (7);
    \draw[-] (7) to node[e] {1} (8);
    \draw[-] (7) to node[e] {1} (9);
    \draw[-] (2) to node[e] {1} (10);
    \draw[-] (2) to node[e] {1}(11);

    \node at (.2,4){$T_1$};
  \end{scope}
  \begin{scope}[shift={(5.5,0)}]
    \node[node, label=west:\textcolor{blue}{$(7)$}] at (-0.3,-0.6) (5) {7};
    \node[node, label=north:\textcolor{blue}{$(3 \;8)$}] at (-1,-1.8) (12) {8};
    \node[node] at (-0.3,.8) (6) {6};
    \node[node, label=west:\textcolor{blue}{$(5)$}] at (-1,2) (7) {5};
    \node[node] at (-0.3,3.2) (9) {3};
    \node[node] at (-1.7,3.2) (8) {4};
    \node[node, label=north:\textcolor{blue}{$(6\; 4\;2 \;1\;9)$}] at (1, -1.8) (2) {9};
    \node[node] at (1,-.4) (10) {2};
    \node[node] at (2.3,-.6) (11) {1};

    \draw[-] (12) to node[e] {1} (5);
    \draw[-] (5) to node[e] {3} (6);
    \draw[-] (6) to node[e] {2} (7);
    \draw[-] (7) to node[e] {1} (8);
    \draw[-] (7) to node[e] {1} (9);
    \draw[-] (2) to node[e] {1} (10);
    \draw[-] (2) to node[e] {1}(11);
    \draw[-] (5) to node[e] {2}(2);

    \node at (.2,4){$T_2$};
  \end{scope}
  \begin{scope} [shift={(11,0)}]
    \node[node, label=west:\textcolor{blue}{$(7)$}] at (-0.3,-0.6) (5) {7};
    \node[node, label=north:\textcolor{blue}{$(3 \;8)$}] at (-1,-1.8) (12) {8};
    \node[node] at (-0.3,.8) (6) {6};
    \node[node, label=west:\textcolor{blue}{$(5)$}] at (-1,2) (7) {5};
    \node[node] at (-0.3,3.2) (9) {3};
    \node[node] at (-1.7,3.2) (8) {4};
    \node[node, label=north:\textcolor{blue}{$(6)(4\;2 \;1\;9)$}] at (1, -1.8) (2) {9};
    \node[node] at (1,-.4) (10) {2};
    \node[node] at (2.3,-.6) (11) {1};

    \draw[-] (12) to node[e] {1} (5);
    \draw[-] (5) to node[e] {3} (6);
    \draw[-] (6) to node[e] {2} (7);
    \draw[-] (7) to node[e] {1} (8);
    \draw[-] (7) to node[e] {1} (9);
    \draw[-] (2) to node[e] {1} (10);
    \draw[-] (2) to node[e] {1}(11);
    \draw[-] (5) to node[e] {2}(2);

    \node at (.2,4){$T_3$};
  \end{scope}
\end{tikzpicture}
\caption{Example of the two possible cases in step (3) of Algorithm~\ref{alg:labelled}. See Example~\ref{ex:algstep3}.}
    \label{fig:alg_ex1}
\end{figure}
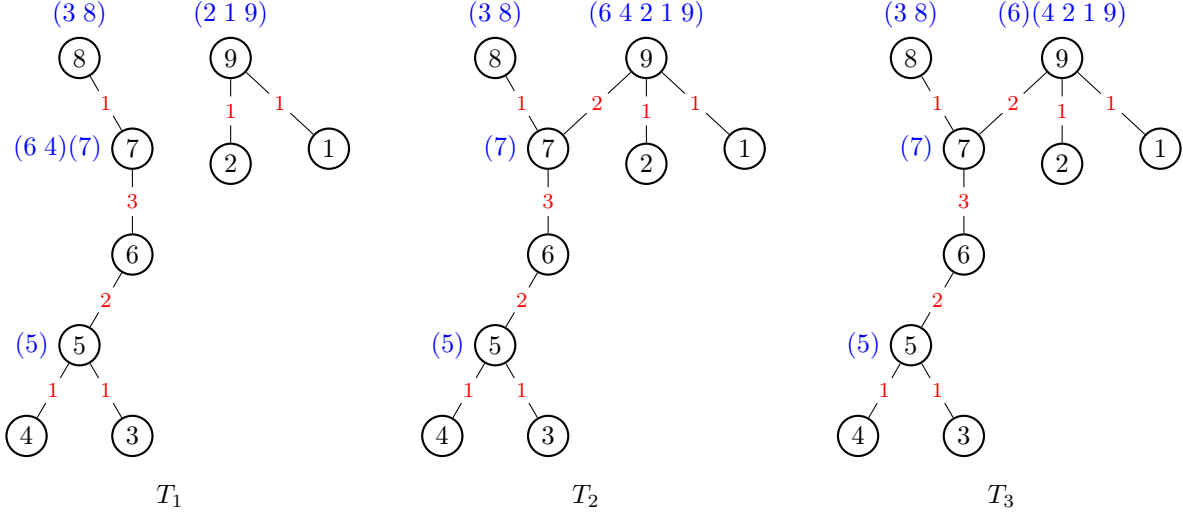

The following proposition outlines some basic properties of Algorithm~\ref{alg:labelled}, in particular verifying that it does output a factorization forest.
\begin{proposition}
    Let $f$ be a minimal factorization of $(\sigma, \pi) \in \allpp_n$, and let $T$ be the labeled graph output by Algorithm~\ref{alg:labelled}.
    \begin{enumerate}[label=(\alph*)]
        \item The connected components of $T$ are given by the parts of $\pi$, and $\sigma = \sigma^{(1)} \cdots \sigma^{(n)}$.
        \item The size of the domain of $\sigma^{(i)}$ is $s(i)$. 
        \item The graph $T$ (ignoring the vertex labels) is a factorization forest. 
    \end{enumerate}
\end{proposition}
\begin{proof}
We may assume that $f$ is in standard form. We will show that the claims hold for each initial segment of $f$---this is trivial for the empty factorization. Suppose that in step (3) we consider a transposition $s_{ij}$. Since $i$ appears in $\sigma^{(k)}$, $i$ and $k$ must lie in the same connected component. Hence adding the edge $\{k,j\}$ merges the connected components containing $i$ and $j$, which matches the change in $\pi$. Also the only change to the product $\sigma^{(1)} \cdots \sigma^{(n)}$ occurs in step (3b), where it gets multiplied by $s_{ij}$, which matches the change in $\sigma$. This proves (a).

For part (b), since $f$ is in standard form, the transpositions before $s_{ij}$ in $f$ all have rank at most $j$. Thus no vertex greater than $j$ has been touched, so $k \leq j$ (as claimed in the algorithm), and $j$ is still in the domain of $\sigma^{(j)}$. Adding a new edge from $k$ to $j$ in step (3a) increases $\outd(k)$ and $\ind(j)$ by $|\alpha|$, and hence it decreases $s(k)$ and increases $s(j)$ by $|\alpha|$. This agrees with the change in the size of the domains of $\sigma^{(j)}$ and $\sigma^{(k)}$ when moving the cycle $\alpha$. This proves part (b) and shows that $s(j) \geq 0$ for all $j$ as required for factorization forests.

All that remains for part (c) is to show that $T$ has no cycles. Suppose that we add the edge $\{k,j\}$ in step (3a) and create a cycle (or parallel edge). Then $j$ and $k$ and hence also $i$ must have been in the same part of $\pi$. But $i$ was in the cycle $\alpha$ of $\sigma^{(k)}$ which does not contain $j$, so $i$ and $j$ were not in the same cycle of $\sigma$. By Lemma~\ref{Lem:minlen_pp}, this implies that $f$ cannot be minimal.
\end{proof}

It follows that if $f$ is a minimal transitive factorization, then its factorization forest is a spanning tree. The next proposition guarantees that this spanning tree is contained in the same quasi-threshold graphs as $f$ is supported on.

\begin{proposition}\label{prop:spanning_tree}
    Let $f$ be a minimal transitive factorization with factorization tree $T$ obtained from Algorithm~\ref{alg:labelled}. Then $f$ is supported on a quasi-threshold graph $G$ if and only if $T$ is a spanning tree of $G$.
\end{proposition}
\begin{proof}
    Since $f$ and its standard form are supported on the same quasi-threshold graphs by Lemma~\ref{lem:std_form}, we may assume that $f$ is in standard form. Let quasi-threshold graphs $G$ and $G'$ be the comparability graphs of rooted forests $F$ and $F'$, and suppose that $f$ is supported on $G$ while $T$ is a subgraph of $G'$. We first show by induction that after each iteration of step (3) in Algorithm~\ref{alg:labelled}, if $u$ lies in the domain of $\sigma^{(v)}$, then $u\preceq v$ in both $F$ and $F'$.

    The claim is trivial at the initialization of $T$. We need to check that if, when we consider a transposition $s_{ij}$ in step (3), we move a cycle $\alpha$ from $\sigma^{(k)}$ to $\sigma^{(j)}$, then the claim remains valid. By transitivity, it suffices to show that $k \prec j$ in both $F$ and $F'$. For $F'$, this holds since $\{k,j\}$ is an edge of $T \subseteq G'$. For $F$, since $i$ lies in the domain of $\sigma^{(k)}$, $i \preceq k$ in $F$. Since $f$ contains $s_{ij}$ and is supported on $G$, $j \succ i$ in $F$. But $F$ is a rooted tree, so $j \succ i$ and $i \preceq k$ (and $j > k$) imply $j \succ k$ in $F$. This proves the claim.
    
    To complete the proof, we need to show that $T$ is a subgraph of $G$ and that $f$ is supported on $G'$. The first holds since $j \succ k$ in $F$ implies that the new edge $\{k,j\}$ of $T$ also lies in $G$. For the second, we have $j \succeq k \succeq i$ in $F'$, so the edge $\{i,j\}$ corresponding to the new transposition $s_{ij}$ is contained in $G'$. 
\end{proof}

By looking at the factorization forest at a vertex $i$, we can analyze the rank $i$ part of the factorization. In the following proposition, let $(\sigma, \pi)$ be the partitioned permutation consisting of the cycles moved along incoming edges at $i$ during the construction of the factorization forest of $f$, all in separate parts. Also let $(\sigma', \pi')$ consist of the cycles moved along outgoing edges at $i$ together with those in the final vertex label $\sigma^{(i)}$, all in a single part.

\begin{proposition} \label{prop:vertexi}
Let $f$ be a minimal factorization in standard form with factorization forest $T$, and let $f_i = (t_1,\dots, t_k)$ be the subsequence of $f$ consisting of the transpositions of rank $i$. Then $(\sigma,\pi) \cdot t_1 \cdots t_k = (\sigma',\pi')$ as in Theorem~\ref{thm:lattice_walk_count}(a), where $(\sigma,\pi)$ and $(\sigma',\pi')$ are as described above.

In particular, if $f$ is a factorization of the identity permutation, then the length of $f_i$ is $l(i) = \deg(i) + s(i) - 1$, where $\deg(i)$ is the unweighted degree of vertex $i$ in $T$.
\end{proposition}
\begin{proof}
    Consider the change to the vertex label $\sigma^{(i)}$ in Algorithm~\ref{alg:labelled} when we consider the transpositions in $f_i$. For each incoming edge $e$ at vertex $i$, a cycle of length $w(e)$ from a different connected component is added to $\sigma^{(i)}$; we may treat these as the single partitioned permutation $(\sigma, \pi)$. This is then multiplied by the transpositions in $f_i$. We know the result is $(\sigma', \pi')$ since later, one cycle is removed along each outgoing edge $e$, resulting in $\sigma^{(i)}$. This therefore matches the situation in Theorem~\ref{thm:lattice_walk_count}(a). Clearly $f_i$ must have minimal length if $f$ is minimal.
    
    For the final claim, use Theorem~\ref{thm:lattice_walk_count}(a), noting that $\sigma^{(i)}$ consists of $s(i)$ $1$-cycles, so $f_i$ has length $c(\sigma)+c(\sigma') - 1 = \deg(i) + s(i) - 1$.
\end{proof}

\subsection{Counting factorizations}

We are now ready to combine the previous results to determine the number of minimal transitive factorizations corresponding to each factorization tree $T$. Summing over the trees that are contained in a given quasi-threshold graph will then yield our main result. 

Recall from above that for a vertex $i$ in a factorization tree $T$, 
\begin{itemize}
\item $\deg(i)$ is the unweighted degree of $i$,
\item $s(i) = 1 + \ind(i) - \outd(i)$, and
\item $l(i) = \deg(i) + s(i) - 1$.
\end{itemize}
Note that if we sum over all vertices $i$,
\begin{align*}
\sum_{i=1}^n \deg(i) &= 2n-2, \\
\sum_{i=1}^n s(i) &= n, \\
\sum_{i=1}^n l(i) &= (2n-2) + n - n = 2n-2.
\end{align*}

Define the \emph{weight} of a factorization tree $T = ([n], E)$ to be
\[ w(T) = \displaystyle{\binom{n}{s(1),\dots,s(n)}} \prod_{e \in E}w(e). \]

\begin{theorem}\label{thm:count_N_id}
    For any factorization tree $T$ on vertex set $[n]$, the number of minimal transitive factorizations $f$ of the identity with factorization tree $T$ is
    \begin{equation*}\label{eq:id_tree_star}
        \numftree{T} = \frac{(2n-2)!}{n!} \cdot w(T).
    \end{equation*}
\end{theorem}

\begin{proof}
    To choose a minimal transitive factorization of the identity with factorization tree $T$, we need to do the following:
    \begin{itemize}
    \item At each vertex $i$ from smallest to largest, given disjoint cycles $(\sigma, \pi)$ assigned to the incoming edges,
        \begin{itemize} 
            \item choose $f_i = (t_1, \dots, t_{l(i)})$ so that $(\sigma',\pi') = (\sigma,\pi) \cdot t_1 \cdots t_{l(i)}$ has the correct cycle type (as determined by the weights of outgoing edges and $s(i)$ using Proposition~\ref{prop:vertexi}); then
            \item assign cycles of $\sigma'$ to outgoing edges at $i$ of the correct weight.
        \end{itemize}
    \item Choose a factorization with standard form $f_1 f_2 \cdots f_n$. 
    \end{itemize}
    Note that since $\sum_i l(i) = 2n-2 = \len(\id)$, our choices of $f_i$ are guaranteed to yield a minimal transitive factorization of the identity.

    For a given vertex $i$, let $\lambda$ be the partition of incoming weights at $i$, and let $\mu$ be the partition of outgoing weights at $i$ combined with $s(i)$ parts of size $1$. Then by Proposition~\ref{prop:vertexi} and Theorem~\ref{thm:lattice_walk_count}(b), the number of choices for $f_i$ is $\frac{l(i)!}{\aut(\mu)} \cdot \Pi(\lambda)$.
    The number of ways to assign the cycles of $\sigma'$ to the outgoing edges at $i$ is $\frac{\aut(\mu)}{s(i)!}$ (since the $s(i)$ $1$-cycles that stay at vertex $i$ are indistinguishable). Therefore the total number of choices at vertex $i$ is $\frac{l(i)!}{s(i)!} \cdot \Pi(\lambda)$. Multiplying over all vertices gives \[\prod_{i=1}^n \frac{l(i)!}{s(i)!} \cdot \prod_{e \in E} w(e)\]
    factorizations in standard form corresponding to $T$. 

    By Lemma~\ref{lem:std_form}, each standard form yields one factorization for each permutation of its rank sequence, which contains $l(i)$ entries equal to $i$. Thus the total number of factorizations is 
    \begin{align*}
    \numftree{T} &= \prod_{i=1}^n \frac{l(i)!}{s(i)!} \cdot \prod_{e \in E} w(e) \cdot \binom{2n-2}{l(1), \dots, l(n)}\\
    &= \frac{(2n-2)!}{n!} \binom{n}{s(1), \dots, s(n)} \cdot \prod_{e \in E} w(e)\\
    &= \frac{(2n-2)!}{n!} \cdot w(T),
    \end{align*}
    as desired.
\end{proof}

We now present our main theorem giving a combinatorial formula for $\numfact{\id}{G}$, the number of minimal transitive factorizations of the identity supported on a quasi-threshold graph $G$.

\begin{theorem}\label{thm:count_quasi_threshold}
    Let $G$ be a quasi-threshold graph on $n$ vertices. Then
    \begin{equation*}\label{eq:main_thm}
        \numfact{\id}{G} = \frac{(2n-2)!}{n!}\sum_{T} w(T),
    \end{equation*}
    where $T$ ranges over factorization trees that are spanning trees of $G$.
\end{theorem}

\begin{proof}
By Proposition~\ref{prop:spanning_tree}, we need only sum Theorem~\ref{thm:count_N_id} over factorization trees that are spanning trees of $G$.
\end{proof}

We present an example calculating all possible factorization trees $T$ supported on $K_4$.

\begin{example}\label{ex:K_4}
Figure~\ref{fig:k4_nid} shows all factorization trees $T_j$ supported on $K_4$ together with their weights $w(T_j)$. (The dots around a vertex $i$ indicate $s(i)$.) For example, for $T_3$: 
\[
    w(T_{3})  = \binom{n}{s(4),s(3)}\displaystyle \prod_{e \in E} w(e) 
     = \binom{4}{2,2} \cdot 2
     = 12.  \]

If we sum $w(T_j)$ over all trees in Figure~\ref{fig:k4_nid}, then we obtain
\[\sum_{j=1}^{16} w(T_j) = 96, \]
and thus we have
\[ \numfact{\id}{K_4} = \frac{(2n-2)!}{n!} \cdot \sum_{j=1}^{16} w(T_j) = 30 \cdot 96 = 2880.\]

Once we have a list of all possible factorization trees for the complete graph $K_n$, for any other quasi-threshold graph $G$ on $n$ vertices, we only need to sum up $N(T_j)$ over the spanning trees of $G$ to obtain $\numfact{\id}{G} = \sum_{T_j} \numftree{T_j}$. For example, if $G$ is the quasi-threshold graph shown in Figure~\ref{fig:add_up_trees}, then it contains factorization trees $T_6$, $T_7$, and $T_{12}$. Since $w(T_6) + w(T_7) + w(T_{12}) = 2 + 4 + 1 = 7$, we get that $\numfact{\id}{G} = \frac{(2n-2)!}{n!} \cdot 7 = 210$.
\end{example}

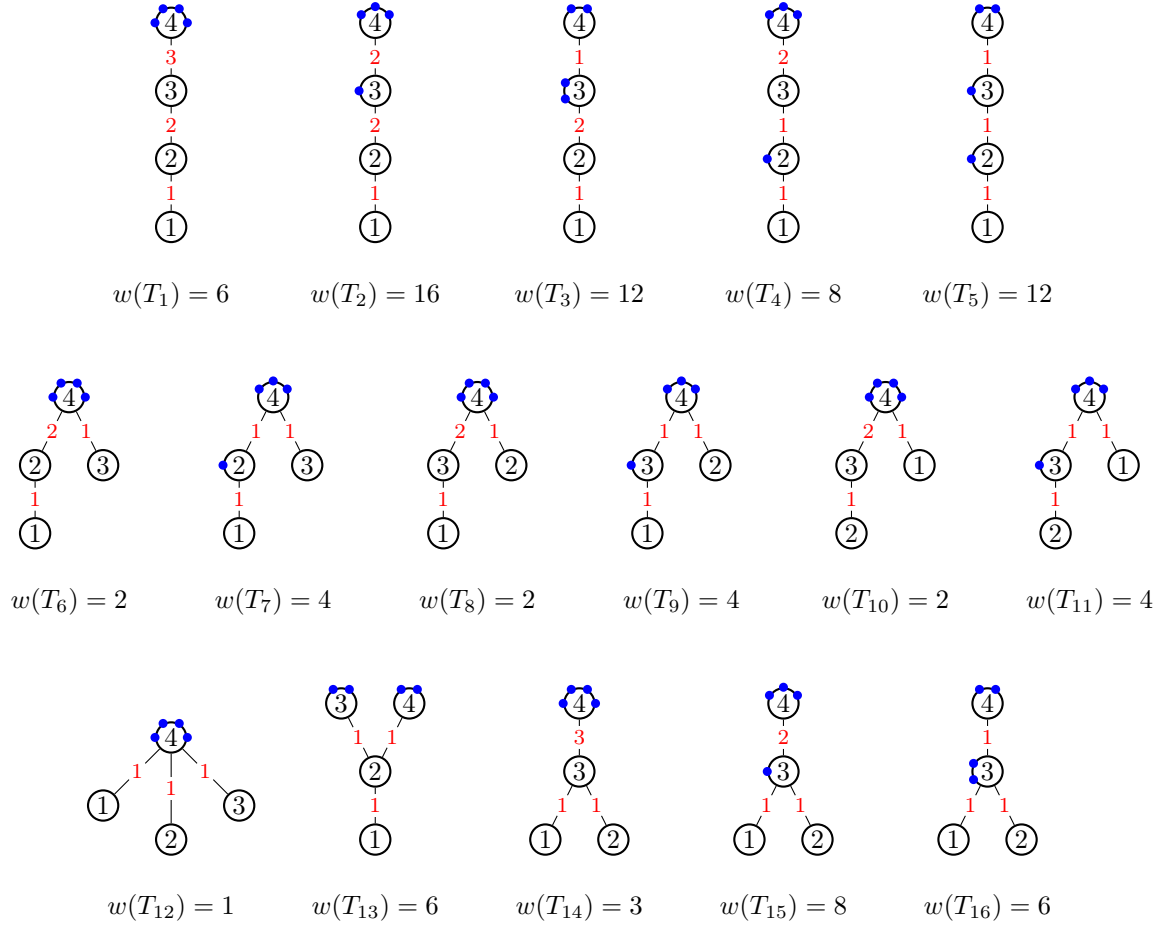
\begin{figure}
    \centering
\begin{tikzpicture}[
    node/.style={circle, draw=black, thick, minimum size=4mm, inner sep=1pt},
    e/.style={fill=white, text=red, circle, inner sep=.7pt, node font=\footnotesize},
    dot/.style={fill=blue, circle, inner sep=1.2pt},
    scale=.9
  ]
  \begin{scope}
    %Nodes
    \node[node] at (0,0) (1) {1};
    \node[node] at (0,1) (2) {2};
    \node[node] at (0,2) (3) {3};
    \node[node] at (0,3) (4) {4};
    %Lines
    \draw[-] (1) to node[e] {1} (2);
    \draw[-] (2) to node[e] {2} (3);
    \draw[-] (3) to node[e] {3} (4);
    \node[dot] at (4.0){};
    \node[dot] at (4.60){};
    \node[dot] at (4.120){};
    \node[dot] at (4.180){};

    \node at (0,-1){$w(T_1)=6$};
  \end{scope}
  \begin{scope}
[shift={(3,0)}]
    %Nodes
    \node[node] at (0,0) (1) {1};
    \node[node] at (0,1) (2) {2};
    \node[node] at (0,2) (3) {3};
    \node[node] at (0,3) (4) {4};
    %Lines
    \draw[-] (1) to node[e] {1} (2);
    \draw[-] (2) to node[e] {2} (3);
    \draw[-] (3) to node[e] {2} (4);
    \node[dot] at (3.180){};
    \node[dot] at (4.30){};
    \node[dot] at (4.90){};
    \node[dot] at (4.150){};

    \node at (0,-1){$w(T_2)=16$};
  \end{scope}
  \begin{scope}
[shift={(6,0)}]
    %Nodes
    \node[node] at (0,0) (1) {1};
    \node[node] at (0,1) (2) {2};
    \node[node] at (0,2) (3) {3};
    \node[node] at (0,3) (4) {4};
    %Lines
    \draw[-] (1) to node[e] {1} (2);
    \draw[-] (2) to node[e] {2} (3);
    \draw[-] (3) to node[e] {1} (4);
    \node[dot] at (3.210){};
    \node[dot] at (3.150){};
    \node[dot] at (4.120){};
    \node[dot] at (4.60){};

    \node at (0,-1){$w(T_3)=12$};
  \end{scope}
  \begin{scope}
[shift={(9,0)}]
    %Nodes
    \node[node] at (0,0) (1) {1};
    \node[node] at (0,1) (2) {2};
    \node[node] at (0,2) (3) {3};
    \node[node] at (0,3) (4) {4};
    %Lines
    \draw[-] (1) to node[e] {1} (2);
    \draw[-] (2) to node[e] {1} (3);
    \draw[-] (3) to node[e] {2} (4);
    \node[dot] at (2.180){};
    \node[dot] at (4.30){};
    \node[dot] at (4.90){};
    \node[dot] at (4.150){};

    \node at (0,-1){$w(T_4)=8$};
  \end{scope}
  \begin{scope}
[shift={(12,0)}]
    %Nodes
    \node[node] at (0,0) (1) {1};
    \node[node] at (0,1) (2) {2};
    \node[node] at (0,2) (3) {3};
    \node[node] at (0,3) (4) {4};
    %Lines
    \draw[-] (1) to node[e] {1} (2);
    \draw[-] (2) to node[e] {1} (3);
    \draw[-] (3) to node[e] {1} (4);
    \node[dot] at (2.180){};
    \node[dot] at (3.180){};
    \node[dot] at (4.60){};
    \node[dot] at (4.120){};

    \node at (0,-1){$w(T_5)=12$};
  \end{scope}
  \begin{scope}
[shift={(-1.5,-4.5)}]
    %Nodes
    \node[node] at (-.5,0) (1) {1};
    \node[node] at (-.5,1) (2) {2};
    \node[node] at (.5,1) (3) {3};
    \node[node] at (0,2) (4) {4};
    %Lines
    \draw[-] (1) to node[e] {1} (2);
    \draw[-] (2) to node[e] {2} (4);
    \draw[-] (3) to node[e] {1} (4);
    \node[dot] at (4.0){};
    \node[dot] at (4.60){};
    \node[dot] at (4.120){};
    \node[dot] at (4.180){};

    \node at (0,-1){$w(T_6)=2$};
  \end{scope}
  \begin{scope}
[shift={(1.5,-4.5)}]
    %Nodes
    \node[node] at (-.5,0) (1) {1};
    \node[node] at (-.5,1) (2) {2};
    \node[node] at (.5,1) (3) {3};
    \node[node] at (0,2) (4) {4};
    %Lines
    \draw[-] (1) to node[e] {1} (2);
    \draw[-] (2) to node[e] {1} (4);
    \draw[-] (3) to node[e] {1} (4);
    \node[dot] at (2.180){};
    \node[dot] at (4.30){};
    \node[dot] at (4.90){};
    \node[dot] at (4.150){};

    \node at (0,-1){$w(T_7)=4$};
  \end{scope}
  \begin{scope}
[shift={(4.5,-4.5)}]
    %Nodes
    \node[node] at (-.5,0) (1) {1};
    \node[node] at (-.5,1) (2) {3};
    \node[node] at (.5,1) (3) {2};
    \node[node] at (0,2) (4) {4};
    %Lines
    \draw[-] (1) to node[e] {1} (2);
    \draw[-] (2) to node[e] {2} (4);
    \draw[-] (3) to node[e] {1} (4);
    \node[dot] at (4.0){};
    \node[dot] at (4.60){};
    \node[dot] at (4.120){};
    \node[dot] at (4.180){};

    \node at (0,-1){$w(T_8)=2$};
  \end{scope}
  \begin{scope}
[shift={(7.5,-4.5)}]
    %Nodes
    \node[node] at (-.5,0) (1) {1};
    \node[node] at (-.5,1) (2) {3};
    \node[node] at (.5,1) (3) {2};
    \node[node] at (0,2) (4) {4};
    %Lines
    \draw[-] (1) to node[e] {1} (2);
    \draw[-] (2) to node[e] {1} (4);
    \draw[-] (3) to node[e] {1} (4);
    \node[dot] at (2.180){};
    \node[dot] at (4.30){};
    \node[dot] at (4.90){};
    \node[dot] at (4.150){};

    \node at (0,-1){$w(T_9)=4$};
  \end{scope}
  \begin{scope}
[shift={(10.5,-4.5)}]
    %Nodes
    \node[node] at (-.5,0) (1) {2};
    \node[node] at (-.5,1) (2) {3};
    \node[node] at (.5,1) (3) {1};
    \node[node] at (0,2) (4) {4};
    %Lines
    \draw[-] (1) to node[e] {1} (2);
    \draw[-] (2) to node[e] {2} (4);
    \draw[-] (3) to node[e] {1} (4);
    \node[dot] at (4.0){};
    \node[dot] at (4.60){};
    \node[dot] at (4.120){};
    \node[dot] at (4.180){};

    \node at (0,-1){$w(T_{10})=2$};
  \end{scope}
  \begin{scope}
[shift={(13.5,-4.5)}]
    %Nodes
    \node[node] at (-.5,0) (1) {2};
    \node[node] at (-.5,1) (2) {3};
    \node[node] at (.5,1) (3) {1};
    \node[node] at (0,2) (4) {4};
    %Lines
    \draw[-] (1) to node[e] {1} (2);
    \draw[-] (2) to node[e] {1} (4);
    \draw[-] (3) to node[e] {1} (4);
    \node[dot] at (2.180){};
    \node[dot] at (4.30){};
    \node[dot] at (4.90){};
    \node[dot] at (4.150){};

    \node at (0,-1){$w(T_{11})=4$};
  \end{scope}
  \begin{scope}
[shift={(0,-9)}]
    %Nodes
    \node[node] at (-1,.5) (1) {1};
    \node[node] at (0,0) (2) {2};
    \node[node] at (1,.5) (3) {3};
    \node[node] at (0,1.5) (4) {4};
    %Lines
    \draw[-] (1) to node[e] {1} (4);
    \draw[-] (2) to node[e] {1} (4);
    \draw[-] (3) to node[e] {1} (4);
    \node[dot] at (4.0){};
    \node[dot] at (4.60){};
    \node[dot] at (4.120){};
    \node[dot] at (4.180){};

    \node at (0,-1){$w(T_{12})=1$};
  \end{scope}
  \begin{scope}
[shift={(3,-9)}]
    %Nodes
    \node[node] at (0,0) (1) {1};
    \node[node] at (0,1) (2) {2};
    \node[node] at (-.5,2) (3) {3};
    \node[node] at (.5,2) (4) {4};
    %Lines
    \draw[-] (1) to node[e] {1} (2);
    \draw[-] (2) to node[e] {1} (3);
    \draw[-] (2) to node[e] {1} (4);
    \node[dot] at (3.60){};
    \node[dot] at (3.120){};
    \node[dot] at (4.60){};
    \node[dot] at (4.120){};

    \node at (0,-1){$w(T_{13})=6$};
  \end{scope}
  \begin{scope}
[shift={(6,-9)}]
    %Nodes
    \node[node] at (-.5,0) (1) {1};
    \node[node] at (.5,0) (2) {2};
    \node[node] at (0,1) (3) {3};
    \node[node] at (0,2) (4) {4};
    %Lines
    \draw[-] (1) to node[e] {1} (3);
    \draw[-] (2) to node[e] {1} (3);
    \draw[-] (3) to node[e] {3} (4);
    \node[dot] at (4.0){};
    \node[dot] at (4.180){};
    \node[dot] at (4.60){};
    \node[dot] at (4.120){};

    \node at (0,-1){$w(T_{14})=3$};
  \end{scope}
  \begin{scope}
[shift={(9,-9)}]
    %Nodes
    \node[node] at (-.5,0) (1) {1};
    \node[node] at (.5,0) (2) {2};
    \node[node] at (0,1) (3) {3};
    \node[node] at (0,2) (4) {4};
    %Lines
    \draw[-] (1) to node[e] {1} (3);
    \draw[-] (2) to node[e] {1} (3);
    \draw[-] (3) to node[e] {2} (4);
    \node[dot] at (3.180){};
    \node[dot] at (4.30){};
    \node[dot] at (4.90){};
    \node[dot] at (4.150){};

    \node at (0,-1){$w(T_{15})=8$};
  \end{scope}
  \begin{scope}
[shift={(12,-9)}]
    %Nodes
    \node[node] at (-.5,0) (1) {1};
    \node[node] at (.5,0) (2) {2};
    \node[node] at (0,1) (3) {3};
    \node[node] at (0,2) (4) {4};
    %Lines
    \draw[-] (1) to node[e] {1} (3);
    \draw[-] (2) to node[e] {1} (3);
    \draw[-] (3) to node[e] {1} (4);
    \node[dot] at (3.150){};
    \node[dot] at (3.210){};
    \node[dot] at (4.60){};
    \node[dot] at (4.120){};

    \node at (0,-1){$w(T_{16})=6$};
  \end{scope}
\end{tikzpicture}

\caption{The $16$ factorization trees on $[4]$ with their weights. Dots indicate the value of $s(i)$ at each vertex $i$.}
\label{fig:k4_nid}

\end{figure}

\begin{figure}
    \centering
  \begin{tikzpicture}[
node/.style={circle, draw=black, thick, minimum size=4mm, inner sep=2.5pt},
]
%Nodes
\node[node] at (0,2.7) (1) {4};
\node[node] at (1,1.8) (4) {3};
\node[node] at (0,1.5) (5) {2};
\node[node] at (-1, 1.8) (3) {1};

%Lines
\draw[-] (1) -- (4);
\draw[-] (1) -- (5);
\draw[-] (1) -- (3);
\draw[-] (3) -- (5);
\end{tikzpicture}
    \caption{A quasi-threshold graph that contains factorization trees $T_6, T_7, T_{12}$ in Figure~\ref{fig:k4_nid}.}
    \label{fig:add_up_trees}
\end{figure}
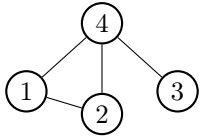

Note that $K_4$ has $16$ factorization trees, which is exactly the number of (ordinary) spanning trees of $K_4$. In fact, a similar result holds for any quasi-threshold graph $G$: the number of terms in the sum in Theorem~\ref{thm:count_quasi_threshold} is equal to the number of spanning trees of $G$. (See also \cite[Theorem 6.1]{floor_diagrams}.)

\begin{proposition} \label{prop:cayley}
For a quasi-threshold graph $G$, the number of factorization trees that are spanning trees of $G$ is equal to the number of (ordinary) spanning trees of $G$. In particular, there are exactly $n^{n-2}$ factorization trees on $n$ vertices.
\end{proposition}
\begin{proof}
We induct on the number of vertices of $G$; the result is trivial if $G$ has only one vertex. We may also assume that $G$ is connected.

One can form a spanning tree of $G$ inductively as follows: partition the set of vertices $1, \dots, n-1$ into parts $V_1, \dots, V_k$ and choose a spanning tree for each part. We can then connect vertex $n$ to exactly one vertex from each part in $|V_1| \cdots |V_k|$ ways.

Similarly, one can form a factorization tree that spans $G$ inductively as follows: partition the vertices as before and choose a factorization tree spanning each part. Since any induced subgraph of $G$ is also quasi-threshold, the number of factorization trees spanning $V_i$ is equal to its number of spanning trees by induction. It therefore suffices to show that there are still $|V_1| \cdots |V_k|$ ways to connect vertex $n$ to each part.

Note that since $G$ is a (connected) quasi-threshold graph, vertex $n$ is adjacent to every other vertex. Adding to the factorization tree an edge of weight $m$ from vertex $i$ to vertex $n$ will decrease $s(i)$ by $m$. Therefore to ensure that $s(i)$ remains nonnegative, we can only take $m = 1, \dots, s(i)$. It follows that for any part $V_j$, there are $\sum_{i \in V_j} s(i) = |V_j|$ ways to choose a weighted edge from a vertex in that part to vertex $n$. These choices can be made independently for each part, so there are $|V_1| \cdots |V_k|$ ways to make these choices in total. The result follows.
\end{proof}

\section{Conclusion}\label{sec:conclusion}

We conclude by providing a conjecture and some further discussion. 

\subsection{Cographs} We first define a broader class of graphs that strictly contains the quasi-threshold graphs. 

\begin{definition}
    A graph $G$ is a \textit{cograph} if it can be generated from single-vertex graphs by the operations of disjoint union, join, and complementation. 
\end{definition}

Cographs have many equivalent definitions: for instance, they are the comparability graphs of series-parallel posets, and they are also the graphs containing no induced path on four vertices. See \cite{cograph} for more information.

In general, $\numfact{\id}{G}$ is not always divisible by $\frac{(2n-2)!}{n!}$ for all graphs $G$. For instance, it is not the case for the path on $4$ vertices; see the appendix for other examples. From computer experimentation on all graphs on at most $7$ vertices with SageMath \cite{sage}, this divisibility relation does appear to hold for all cographs, as stated in the following conjecture. 

\begin{conjecture} \label{conj:cograph}
    When $G$ is a cograph, $\numfact{\id}{G}$ is divisible by $\frac{(2n-2)!}{n!}$.
\end{conjecture}

Even though we only conjecture the forward direction of Conjecture~\ref{conj:cograph}, from our experimental data on graphs on at most $7$ vertices, the converse also seems to be true: the divisibility relation does not hold for any connected graph that is not a cograph. It would be interesting to see if the methods of this paper such as factorization trees can be generalized to count the number of factorizations supported on any cograph. %It is perhaps also worth noting that, in the group algebra $\mathbb C[S_n]$, the sum of the transpositions supported on a cograph has only integer eigenvalues, though it is not clear whether this relates to the desired divisibility condition.

\subsection{Tesler matrices}

For any factorization forest $T$, we can construct an upper triangular matrix $A = (a_{ij})_{1 \leq i \leq j \leq n}$ by setting $a_{ii} = s(i)$ and, for $i < j$, $a_{ij} = w(i,j)$ (or $0$ if this edge does not lie in $T$). The resulting matrix then has the property that, for every $i$, \[\sum_{j \geq i} a_{ij} - \sum_{j<i} a_{ji} = 1.\] Nonnegative integer matrices with this property are called \emph{Tesler matrices} \cite{Haglund}. Tesler matrices can also be viewed as the lattice points in a \emph{flow polytope} for $K_{n+1}$ with netflow vector $(1,1,\dots, 1, -n)$ \cite{tesler_polytope}. A formula of Haglund \cite{Haglund} expresses the Hilbert series of the space of diagonal harmonics (which has dimension $(n+1)^{n-1}$) as a $q,t$-weighted sum over Tesler matrices. Although the weight described by Haglund does not quite specialize to the definition of $w(T)$ that we give above, it would be interesting to see if our results have any connection to this work (or related work such as \cite{flow_threshold}), or if our $w(T)$ has a $q$-analogue with notable properties.

\subsection{Labeled floor diagrams}

Factorization trees are equivalent to genus 0 labeled floor diagrams as defined by Fomin and Mikhalkin \cite{floor_diagrams}, who used them to give a combinatorial interpretation of certain Gromov-Witten invariants. We are not aware of any direct connection to the present work, as the weights we define here are not the same as the weights that they use. Floor diagrams are also defined for any genus, so it would be interesting if these had an application to counting, for instance, nonminimal factorizations.

\subsection{Closed formulas}
Although we have given a combinatorial formula for $\numfact{\id}{G}$ in terms of factorization trees, it is not clear how to quickly evaluate this formula in most cases or whether there exists a more succinct formula for this sum. For example, from Theorem~\ref{thm:Hurwitz}, we know that the total weight of all factorization trees on $n$ vertices is $n^{n-2} (n-1)!$. In other words, the average weight of a factorization tree on $n$ vertices is exactly $(n-1)!$. We do not know of an independent combinatorial proof of this fact, which would effectively give a combinatorial proof for Hurwitz's result that $\numfact{\id}{K_n} = n^{n-3} (2n-2)!$.

Through computer experimentation, one can also conjecture formulas for other families of quasi-threshold graphs or cographs, such as
\begin{align*}
\numfact{\id}{K_n - e} &= n^{n-5} (2n-2)!  \cdot \frac{(n-2)(2n-3)}{2}, \\
\numfact{\id}{K_{1,1,n-2}} &= \frac{(2n-2)!}{n!} \cdot \frac{(2n-3)!}{(n-2)!^2}, \\
\numfact{\id}{K_{2,n-2}} &= \frac{(2n-2)!}{n!} \cdot \left(n \binom{2n-3}{n-1} - 2^{2n-3}\right),
\end{align*}
though again it is unclear how or whether the results in this paper can be used to help derive such formulas.

\section{Acknowledgments}

We thank Jun Xing Go for help with coding, and Tianle Li and Runchi Tan for their helpful discussions.

\addcontentsline{toc}{chapter}{\numberline{}References}

\printbibliography[title={References}]
\newpage
\appendix
\section*{Appendix}
We provide computational data for all connected graphs $G$ on 2--5 vertices in the following table along with $\numfact{\id}{G}$. For $G$ such that $\numfact{\id}{G}$ is divisible by $\frac{(2n-2)!}{n!}$, we include the quotient $\numfact{\id}{G}/\frac{(2n-2)!}{n!}$ in the third column. In the fourth column, we indicate whether $G$ is a cograph with ``Y'' or ``N'' and mark quasi-threshold graphs with an additional *. 

\tikzset{every picture/.append style={scale=0.55}}

\begin{table}[h!]
    \centering
    \begin{tabular}{ P{2.5cm}P{1.3cm}P{1.3cm}P{1.4cm}||P{2.5cm}P{1.3cm}P{1.3cm}P{1.4cm}}
        $G$ & $\numfact{\id}{G}$ & \resizebox{.9\hsize}{!}{$\displaystyle\frac{\numfact{\id}{G}}{(2n-2)!/n!}$} & Cograph & $G$ & $\numfact{\id}{G}$ & \resizebox{.9\hsize}{!}{$\displaystyle\frac{\numfact{\id}{G}}{(2n-2)!/n!}$} & Cograph  \\ [3ex] 
        \hline 
        \hline
        &&&&&&&\\

        		\dr{(0,0) node[v]{} -- (1,0)node[v]{}}
		&1 & 1&Y* &     

        %~D_4 plus edge
		\dr{(0,0)node[v]{} -- ++(-150:1)node[v]{} -- ++(-150:1)node[v]{} -- ++(90:1)node[v]{} -- ++(-30:1) -- ++(-30:1)node[v]{}}
		 &2688 & 8 & Y* \\

		%2-path
		\dr{(0,0) node[v]{} -- (1,0)node[v]{} -- (2,0)node[v]{}}
		&4 & 1 &  Y*  & 
        %pentagon
		\dr{(0,0)node[v]{}--++(36:1)node[v]{}--++(-36:1)node[v]{}--++(-108:1)node[v]{}--++(-1,0)node[v]{}--cycle}
    &4480 & - & N \\

		%triangle
		\dr{(0,0) node[v]{} -- (1,0)node[v]{} -- +(120:1)node[v]{} -- cycle}
		&24 & 6 & Y* &
        %diamond+top edge
		\dr{(2,0)node[v]{}--(1,0)node[v]{}--(0,0)node[v]{}--(0,1)node[v]{}--(1,1)node[v]{}--(1,0)node[v]{} (1,0)--(0,1)}
         &12768 & 38 & Y* \\
		
        %3-path
		\dr{(0,0) node[v]{} -- (1,0) node[v]{} -- (2,0) node[v]{} -- (3,0) node[v]{}}
		&42 & - & N & 
		%diamond + side edge
		  \dr{(2,0)node[v]{}--(1,0)node[v]{}--(0,0)node[v]{}--(0,1)node[v]{}--(1,1)node[v]{}--(1,0)node[v]{} (0,0)--(1,1)}  &15616 & - & N \\
		%3-star = D_4
		\dr{(0,0) node[v]{} -- +(150:1) node[v]{}  +(-150:1) node[v]{}--(0,0)--(1,0) node[v]{}}
		&30 & 1 & Y* & %house
		\dr{(0,0)node[v]{}--(-1,0)node[v]{}--(-1,1)node[v]{}--(0,1)node[v]{}--(30:1)node[v]{}--(0,0)--(0,1)node[v]{}}
		&17448 & - & N \\
 
    %triangle+edge
		\dr{(0,0) node[v]{} -- +(150:1) node[v]{} -- +(-150:1) node[v]{}--(0,0)--(1,0) node[v]{}}
		&210 & 7 & Y* & 
    %bowtie
		\dr{(0,0)node[v]{}--++(150:1)node[v]{}--++(0,-1)node[v]{}--++(30:1)--++(30:1)node[v]{}--++(0,-1)node[v]{}--cycle}
    &19824 & 59 & Y*\\

        %square
		\dr{(0,0)node[v]{}--(1,0)node[v]{}--(1,1)node[v]{}--(0,1)node[v]{}--cycle}
		&240 & 8 & Y & 
\dr{(0,0)node[v]{}--(-1,0)node[v]{}--(0,1)node[v]{}--(0,0)--(-1,1)node[v]{}--(0,1)--(30:1)node[v]{}--(0,0)}
		&47040 & 140 & Y* \\
  
        %square+diagonal
		\dr{(0,0)node[v]{}--(1,0)node[v]{}--(1,1)node[v]{}--(0,1)node[v]{}--cycle--(1,1)}
		&900 & 30 & Y* &
          %K_2,3
		\dr{(0,0)node[v]{}--(1,0)node[v]{}--(1,1)node[v]{}--(0,1)node[v]{}--cycle--(1,1) (.5,.5)node[v]{}}
		&15792 & 47 & Y \\
        
        %K_4
		\dr{(0,0)node[v]{}--(1,0)node[v]{}--(1,1)node[v]{}--(0,1)node[v]{}--cycle--(1,1) (1,0)--(0,1)}
		&2880 & 96 & Y* & 
  %K4+edge
		\dr{(1,0)node[v]{}--(0,0)node[v]{}--(-1,0)node[v]{}--(-1,1)node[v]{}--(0,1)node[v]{}--(0,0)--(-1,1) (-1,0)--(0,1)}
		&44016 & 131 & Y* \\
        %4-path
		\dr{(0,0)node[v]{}--(0.9,0)node[v]{}--(1.8,0)node[v]{}--(2.7,0)node[v]{}--(3.6,0)node[v]{}}
		&816 & - & N & %K_2,3 plus an edge
		\dr{(0,0)node[v]{}--(1,0)node[v]{}--(1,1)node[v]{}--(0,1)node[v]{}--cycle--(1,1) (.5,.5)node[v]{}--(1,0)}
		&56784 & 169 & Y\\
        %D_5
		\dr{(0,0) node[v]{} -- +(150:1) node[v]{} +(-150:1) node[v]{}--(0,0)--(1,0) node[v]{}--(2,0)node[v]{}}
		&608 & - & N &
        %gem
		\dr{(120:1)node[v]{}--(-1,0)node[v]{}--(0,0)node[v]{}--(120:1)--(60:1)node[v]{}--(1,0)node[v]{}--(0,0)--(60:1)}
		&57400 & - & N\\
        
		%4-star = ~D_4
		\dr{(0,0)node[v]{} -- ++(45:1)node[v]{} -- +(45:1)node[v]{} +(135:1)node[v]{} -- +(0,0) -- +(-45:1)node[v]{}}
		&336 & 1 & Y* &
		%house+diagonals
		\dr{(-1,0)node[v]{}--(-1,1)node[v]{}--(0,1)node[v]{}--(30:1)node[v]{}--(0,0)node[v]{}--(-1,0)--(0,1)--(0,0)--(-1,1)}
		&157920 & 470 & Y*\\

		%D_5 plus edge
		\dr{(0,0)node[v]{} -- ++(150:1)node[v]{} -- ++(-90:1)node[v]{} -- (0,0) -- (1,0)node[v]{} -- (2,0)node[v]{}}
		&4192 & - & N &
        %box with center
		\dr{(0,0)node[v]{}--(1,0)node[v]{}--(1,1)node[v]{}--(0,1)node[v]{}--(0,0)--(1,1) (1,0)--(0,1) (.5,.5)node[v]{}}
		&168000 & 500 & Y\\
		
		%bull
		\dr{(0,0)node[v]{}--(-1,0)node[v]{}--+(-150:1)node[v]{}--(-1,-1)node[v]{}--(-1,0)(-1,-1)--(0,-1)node[v]{}}    &3480 & - & N &
		%20
		\dr{(306:1)node[v]{}--(18:1)node[v]{}--(90:1)node[v]{}--(162:1)node[v]{}--(234:1)node[v]{}
			--(18:1)--(162:1)--(306:1)--(90:1)--(234:1)}
		&423360 & 1260 & Y*\\

		\dr{(0,0)node[v]{}--(0,1)node[v]{}--(-1,1)node[v]{}--(-1,0)node[v]{}--(0,0)--(1,0)node[v]{}}  
    &3856 & - & N &
		%21
		\dr{(306:1)node[v]{}--(18:1)node[v]{}--(90:1)node[v]{}--(162:1)node[v]{}--(234:1)node[v]{}
			--(18:1)--(162:1)--(306:1)--(90:1)--(234:1)--(306:1)}
		&1008000 & 3000 & Y*\\

    \end{tabular}
\end{table}

\end{document}